\documentclass[review,hidelinks,onefignum,onetabnum]{siamart220329}
\usepackage{listings}%
\usepackage{amsmath,amssymb}
\usepackage{enumitem}
\usepackage{bm}
\usepackage{cite}
\usepackage{float}

\newcommand{\p}{\partial}

\newcommand{\ds}{\displaystyle}

\usepackage{graphicx}
\usepackage{subfig}  
\usepackage{color}

\newtheorem{remark}{Remark}[section]
\headers{Multiple solution computation of elliptic problems}{Lei Zhang, Xiangcheng Zheng, Shangqin Zhu}

\title{Numerical analysis for saddle dynamics of some semilinear elliptic problems \thanks{Submitted to... \funding{This work was partially supported by the National Natural Science Foundation of China (No. 12225102, T2321001, 12288101, 12301555), the Natural Science Foundation of Shandong Province (ZR2025QB01), the Taishan Scholars Program of Shandong Province (No. tsqn202306083), the National Key R\&D Program of China (No. 2023YFA1008903).}}
}

\author{Lei Zhang\thanks{Beijing International Center for Mathematical Research, Center for Machine Learning Research, Center for Quantitative Biology, Peking University, Beijing 100871, China,
  (\email{zhangl@math.pku.edu.cn}).}
\and Xiangcheng Zheng\thanks{School of Mathematics, State Key Laboratory of Cryptography and Digital Economy Security, Shandong University, Jinan, 250100, China, 
  (\email{xzheng@sdu.edu.cn}).}
\and Shangqin Zhu\thanks{School of Mathematics, Shandong University, Jinan, 250100, China, 
  (\email{sqzhu@mail.sdu.edu.cn}).}}

\usepackage{amsopn}

\begin{document}
\nolinenumbers
\maketitle

% REQUIRED
\begin{abstract}
This work presents a numerical analysis of computing transition states of semilinear elliptic partial differential equations (PDEs) via the index-1 saddle dynamics, or equivalently, the gentlest ascent dynamics. To establish clear connections between saddle dynamics and numerical methods of PDEs, as well as improving their compatibility, we first propose the continuous-in-space formulation of saddle dynamics for semilinear elliptic problems. This formulation yields a parabolic system that converges to saddle points. We then analyze the well-posedness, $H^1$ stability and error estimates of semi- and fully-discrete finite element schemes. Significant efforts are devoted to addressing the coupling, gradient nonlinearity, nonlocality of the proposed parabolic system, and the impacts of retraction due to the norm constraint. The error estimate results demonstrate the accuracy and index-preservation of the discrete schemes.  
\end{abstract}

% REQUIRED

% REQUIRED

\section{Introduction}
Computing multiple solutions of nonlinear partial differential equations (PDEs) is an important but challenging topic, cf. the review \cite{Zho}. The semilinear elliptic equation with some nonlinear function $f(u)$ 
\begin{equation}\label{elliptic}
	 \Delta u(x)+ f(u(x))=0,~x\in \Omega;~~  u=0,~x\in \p\Omega, 
\end{equation}
where $\Omega$ is a bounded domain in $\mathbb{R}^d$ ($1\leq d\leq 3$) with a  smooth boundary $\partial \Omega$, is a typical multiple solution problem \cite{Bao,Wan,Wan2}. Sophisticated investigations for this problem or its variants have been conducted, including the mountain pass method \cite{Cho}, homotopy method \cite{Hao}, high-linking algorithm \cite{Din},  search-extension method \cite{XieIMA}, the iterative minimization formulation \cite{Gao}, deflation algorithm \cite{Farr}, minimax-type methods \cite{Li2001,Li2002,LiJi,LiuXie,Liu2,Xie0,Xie,Yao} and dimer-type methods \cite{EZho,Dimer,Lev,Qua}. 

High-index saddle dynamics  \cite{YinSISC} is a representative dimer-type method for locating any-index stationary points and constructing solution landscapes,  with successful applications in various fields \cite{Shi,Shi2,wang2021modeling,Yin2020nucleation,YinPRL}.
This work focuses on the index-1 saddle dynamics (I-1 SD), which is also known as gentlest ascent dynamics \cite{EZho} (see \cite{GuZho} for the equivalence). This method locates transition states that connect different minimizers and thus attracts extensive attention.
 Specifically, given a twice Fr\'echet differentiable energy function $E(y)$ with the position variable $ y \in \mathbb{R}^l $ ($1\leq l\in\mathbb N$), a point $y^*$ is called a nondegenerate index-1 saddle point (or a transition state) of $E(y)$ if $\nabla E(y^*)=0$ and $\nabla^2 E(y^*)$ has only one eigenvalue with negative real part and no eigenvalue with vanishing real part. It is worth mentioning that such definition remains valid for  infinite-dimensional problems, as we will consider in this work. Then the I-1 SD of locating nondegenerate index-1 saddle points of $E$ takes the following form \cite{EZho,ZhaSISC}
\begin{equation}\label{SD}
\left\{\begin{aligned}
 & y_t= \mathcal{L}_v(F(y)),&\mathcal{L}_v(\phi):=\beta(\phi- 2v(v,\phi));\\[0.05in]
 & v_t= \mathcal{N}_v(\nabla F(x)v),&\mathcal{N}_v(\phi):=\gamma(\phi- v(v,\phi)),
    \end{aligned}
    \right.
\end{equation}
%\begin{equation} 
%    \left\{
%    \begin{aligned}
%   \beta^{-1} \dot x & =  \big(F(x)-2v(v,F(x))\big),\\[0.05in]
%    \gamma^{-1}\dot v &=  \big(\nabla F(x)v -v(v, \nabla F(x)v )\big),
%    \end{aligned}
%    \right.
%\end{equation}
where $v\in \mathbb{R}^l $ is directional variable, $\beta ,\gamma >0$ are relaxation parameters,  $ F(y) = -\nabla E(y) $ and $ \nabla F(y) = -\nabla^2 E(y) $ are force and negative Hessian matrix, respectively, and $(\cdot,\cdot)$ denotes the inner product of two vectors. When using I-1 SD to search for multiple solutions of PDEs such as (\ref{elliptic}), the force $F$ in (\ref{SD}) is obtained from spatial discretization of the PDE. Consequently, the dimension $l$ of I-1 SD implicitly depends on the number of the degree of freedom of spatial discretization.

 There exist some recent works on numerical analysis for time-discretization methods of I-1 SD \cite{Z3,Z3CSIAM}, where the $F$ and $\nabla F$ are supposed to be globally Lipschitz continuous. This assumption eliminates  potential difficulties caused by unbounded operators such as $\Delta$ in PDE models and we could then focus the attention on time discretization analysis. However, rigorous numerical analysis considering space-time discretization for I-1 SD in computing multiple solutions of PDEs such as (\ref{elliptic}) remains untreated. Indeed, the dynamical system formulation of I-1 SD  prevents its connection with numerical methods of PDEs, although the $F$ in (\ref{SD}) could be generated from spatial discretization of PDE models. Furthermore, though the dimension of I-1 SD (\ref{SD}) for PDE models depends on spatial discretization, the index is an inherent property of the saddle point and should remain invariant. Whether the index could be preserved in numerical scheme is physically important, while the time discretization analysis could not answer this question.

To give a clearer perspective for connections between saddle dynamics and numerical methods of PDEs, we adopt a different approach from the conventional procedure of  ``first discretize PDEs, then invoke saddle dynamics''. Instead, we first formulate the continuous-in-space version of saddle dynamics for PDE problems and then apply suitable numerical discretization methods. One advantage of recovering the continuous-in-space saddle dynamics is that it is feasible to engage with appropriate spatial discretization methods according to the features of PDE models. This approach fully leverages existing research on numerical methods for PDEs and enhances their compatibility with saddle dynamics when computing multiple solutions of PDEs.

As a start of this idea, we consider problem (\ref{elliptic}), which implies  $F(u)=\Delta u+ f(u)$, 
 $\nabla F(u)v=\Delta v +f'(u)v $ and $(g,\tilde g):=\int_\Omega g(x)\tilde g(x)dx$ for the continuous-in-space case such that the continuation of (\ref{SD}) in space results in a coupled parabolic system
\begin{equation}\label{1order}
\begin{array}{l}
   u_t=\mathcal L_v\big(\Delta u +f(u)\big)=\beta(\Delta u +f(u)- 2v(v,\Delta u +f(u))),\\[0.1in] v_t=\mathcal N_v\big(\Delta v+f'(u)v\big)=\gamma(\Delta v+f'(u)v- v(v,\Delta v+f'(u)v)),
   \end{array}
\end{equation}
on $(x,t)\in \Omega\times\mathbb R^+$, equipped with the following initial and boundary conditions
$$u(x,0)=u_0(x),~~v(x,0)=v_0(x),~~x\in\Omega;~~u(x,t)=v(x,t)=0,~~x\in \p\Omega,~~t\geq 0. $$
For this continuous-in-space I-1 SD, we derive the following results:
\begin{itemize}
\item[$\blacktriangle$]  We prove that a point is an index-1 saddle point of the semilinear elliptic problem (\ref{elliptic}) if and only if it is a stationary point of the parabolic system (\ref{1order}) and the solutions to its linearized version at this point converge exponentially, which provides theoretical supports for the effectiveness of (\ref{1order}) in determining index-1 saddle points. 

\item[$\blacktriangle$] We prove the uniqueness, $H^1$ stability and error estimates of the solutions to the spatial semi-discrete scheme of the system (\ref{1order}), which in turn leads to the well approximation to the index-1 saddle point and provides an answer for the index-preservation issue of numerical methods for multiple solutions of semilinear elliptic problems (cf. Theorem \ref{thm33} and Corollary \ref{cor1}). 

\item[$\blacktriangle$] In fully-discrete scheme, a retraction is applied at each step to ensure the norm constraint $\|v\|_{L^2(\Omega)}=1$, which could avoid the failure of SD to locate saddle points (see \cite[Fig. 1]{MiaCSIAM} for an example). Several efforts are devoted to accommodate its impacts:
\begin{itemize}
 \item[$\bullet$] We first prove the existence, uniqueness and $H^1$ stability of numerical solutions. As the well-posedness proof and the  $H^1$ stability estimates are coupled, an induction procedure is adopted (cf. Theorem \ref{stable}). Furthermore, several techniques are utilized in $H^1$ stability estimates such as the high-order perturbation of the gradient of numerical solutions before and after normalization (cf.  (\ref{yx1}));
 \item[$\bullet$] We propose a normalized projection to show that the error could be perturbed by a cubic term (cf. Lemma \ref{li}), which is critical for error estimates. The truncation errors involving the normalized projection are analyzed, where the derivations are carefully carried out to avoid possible order reduction caused by, e.g. the gradient nonlinearity. 
 \item[$\bullet$] In error estimates, an induction procedure is adopted to account for the nonlinearity of the error equation. Different from the conventional induction, where the error equation is usually given a priori, the induction hypothesis in this work is used to justify the validity of the error splitting in Lemma \ref{li} such that the error equation is further modified to facilitate the induction. Again, the error estimates implies the well approximation to the index-1 saddle point and the index-preservation of fully-discrete scheme, cf. Theorem \ref{thm44} and Remark \ref{rem4}. 
 \end{itemize}
 \end{itemize}

The rest of the work is organized as follows: In Section 2 we perform mathematical analysis for the continuous-in-space I-1 SD (\ref{1order}). In Sections 3--4 we investigate semi- and fully-discrete finite element schemes, respectively. Numerical experiments are performed in Section 5 to substantiate the theoretical findings, and a concluding remark is presented in the last section. 

\section{Mathematical analysis}%The index-1 saddle dynamics (or equivalently the gentlest ascent dynamics \cite{EZho,GuZho}) is originally proposed for finding the transition states of the energy functions.  Specially,  It is shown in, e.g. \cite{YinSISC} that a linear stable steady state of (\ref{SD}) is an index-1 saddle point, which indicates the effectiveness of (\ref{SD}).

%In practice, there are numerical examples in \cite{EZho,YinSISC,ZhaSISC} locating transition states of semilinear elliptic problems like (\ref{elliptic}), which first discretize the spatial operators by, e.g. finite difference method, to reduce the problems to finite-dimensional systems and then invoke them in (\ref{SD}) for computation. In other words, 
Let $L^p(\Omega)$ and $W^{m,p}(\Omega)$ for $0\leq m\in\mathbb N$ and $1\leq p\leq \infty$ be standard Sobolev spaces equipped with standard norms \cite{Ada}. In particular, we set $H^m(\Omega):=W^{m,2}(\Omega)$ and $H^m_0(\Omega)$ denotes the closure of $C^\infty_0(\Omega)$, the space of infinitely differentiable functions with compact support in $\Omega$,  with respect to the norm $\|\cdot\|_{H^m(\Omega)}$. For a Banach space $\mathcal X$ and some $T>0$, the Bochner space
$L^p(0,T;\mathcal X)$ contain functions $g$ that are finite under the norm $\|g\|_{L^p(0,T;\mathcal X)}:=\|\|g\|_{\mathcal X}\|_{L^p(0,T)}$ \cite[Definition 1.2.15]{Hyt}.
Then the space $W^{m,p}(0, T;\mathcal X)$ contains functions that are finite under the  norm $\|u\|_{W^{m,p}(0, T;\mathcal X)}:= \sum_{k=0}^m  \|u^{(k)}_t\|_{L^p(0,T;\mathcal X)}$ \cite[Definition 2.5.4]{Hyt}. 
% For a Banach space $\mathcal X$ and some $T>0$, let 
%be the space of functions in $W^{m,p}(0, T)$ with respect to $\|\cdot\|_{\mathcal X}$. 
For simplicity, we denote $\|\cdot\|:=\|\cdot\|_{L^2(\Omega)}$ and omit $\Omega$ in notations of spatial norms.

%\begin{remark}
%If we discretize the inner products in (\ref{1order}) by certain quadrature rules, then (\ref{1order}) degenerates to the form of dynamical system as (\ref{SD}). In other words, the SD (\ref{SD}) is indeed a spatial finite difference scheme of the continuous-in-space SD for infinite-dimensional problems such as semilinear elliptic PDEs. 
%\end{remark}

 Now we turn to the analysis of the continuous-in-space I-1 SD (\ref{1order}). Following \cite{YinSISC}, the initial value of $v$ is selected such that $\|v_0\|=1$. Suppose $v\in H^2(\Omega)$ and $f'(u)\in L^2(\Omega)$ such that the dynamics of $v$ in (\ref{1order}) is well defined under the $L^2$ sense. Then
\begin{equation*}
		\frac{d}{dt}\big(\|v\|^2-1\big) =2(v_t, v)=2\big(\mathcal N_v\big(\Delta v+f'(u)v\big),v\big) =2 \gamma\big(1-\|v\|^2\big)\big(\Delta v+f'(u)v, v\big).
\end{equation*}
 Let $\rho(t)=\|v\|^2-1$ and $g(t)=-2\gamma\big(\Delta v+f'(u)v, v\big)$ such that the above equation becomes $\displaystyle \frac{d}{dt}\rho(t)=\rho(t)g(t)$, that is, $\rho(t)=\rho(0)\exp\big(\int _0^tg(\tau)d\tau\big)$. Since $\rho(0)=0$, then $\rho(t)\equiv 0$ for $t\geq 0$, that is, the I-1 SD (\ref{1order}) has a norm constraint $\|v\|=1$ for $t\geq 0$.

Now we will indicate the effectiveness of (\ref{1order}) in locating index-1 saddle points by the following theorem.
\begin{theorem}\label{thm21}
Assume that the points $(u^*,{v^*_1})\subset H^2(\Omega)\cap H^1_0(\Omega)$ satisfy $\|v^*_1\|=1$ and the operator $-\Delta-f'(u^*)$ under zero Dirichlet boundary conditions has non-zero eigenvalues $\lambda_1^*\leq  \lambda_2^*\leq \lambda_{3}^*\leq \cdots$.  Then the following two statements are equivalent:
	\begin{itemize}
	\item[(a)] $(u^*, v^*_1)$ is a stationary point of (\ref{1order}) and the linearized problem 
	\begin{equation}\label{li1}
   \tilde u_t=\mathcal L_{v_1^*}\big(\Delta\tilde u +f(u^*)+f'(u^*)(\tilde u-u^*)\big),
\end{equation}
	 with an initial condition and zero Dirichlet boundary conditions, admits exponential convergence $\|\tilde u-u^*\|\sim e^{-\sigma t}$ for some $\sigma >0$ and for $t$ large enough; 
	 \item[(b)] $u^*$ is an index-1 saddle point of (\ref{elliptic}) and $v^*_1$ is an eigenfunction of the operator $-\Delta-f'(u^*)$ under zero Dirichlet boundary conditions, with the corresponding eigenvalue $\lambda_1^*<0$.
	\end{itemize}
\end{theorem}
\begin{proof}
	Suppose (a) holds, then $\mathcal L_{v_1^*}\big(\Delta u^* +f(u^*)\big)=0$. As  $\|\mathcal L_{v_1^*}\big(\Delta u^* +f(u^*)\big)\|=\beta \|\Delta u^* +f(u^*)\|$, then $u^*$ is a stationary point of (\ref{elliptic}). As $\mathcal{N}_{v_1^*}(\Delta v^*_1+ f'(u^*)v^*_1)=0$, we have $(\Delta +f'(u^*))v^*_1=\mu^*_1v^*_1$ where $\mu^*_1=(v^*_1, \Delta v^*_1+f'(u^*)v^*_1)$, that is, $(-\mu^*_1, v^*_1)$ is an eigenpair of $-\Delta -f'(u^*)$. Denote other eigenvalues of $-\Delta -f'(u^*)$ as $-\mu^*_2\leq  -\mu^*_3 \leq \cdots $ with corresponding eigenfunctions $\{v_i^*\}_{i=2}^\infty$ such that $\{v_i^*\}_{i=1}^\infty$ form an orthonormal basis of $L^2(\Omega)$. Define $\delta_u:=\tilde {u}-u^*$ such that (\ref{li1}) implies $\delta_{u, t}=\mathcal{L}_{v_1^*}\big(\Delta \delta_u+f'(u^*)\delta_u \big)$. Then we expand $\delta_u$ as $\delta_u=\sum_{i=1}^\infty \delta_{u,i}(t)v_i^*(x)$ such that 
\begin{equation}\label{for}
\begin{aligned}
\delta_{u,t}& =\sum_{i=1}^\infty \delta'_{u,i}(t)v_i^*(x)=\sum_{i=1}^\infty \delta_{u, i}(t)\mathcal{L}_{v_1^*}(\Delta v_i^*+f'(u^*)v_i^*)\\
&=\sum_{i=1}^\infty \delta_{u, i}(t)\mathcal{L}_{v_1^*}(\mu _i^* v_i^*)=-\beta\delta_{u,1}(t)\mu _1^* v_1^*+\beta\sum_{i=2}^{\infty}\delta_{u,i}(t)\mu _i^* v_i^*.
\end{aligned}
\end{equation} 
We then solve the ordinary differential equations of $\{\delta_{u, i}\}_{i=1}^\infty$ to obtain $\delta_{u,1}(t)=(\delta _u(0),v_1^*)e^{-\beta \mu_1^*t}$ and $\delta_{u,i}(t)=(\delta _u(0),v_i^*)e^{\beta \mu_i^*t}$  for $i\geq 2$. If $\mu_1^*\leq 0$, we could select $\delta_u(0)=v^*_1$ such that $\|\delta_{u}\|=e^{-\beta \mu_1^*t}$, a contradiction to the assumption $\|\delta_u\|\sim e^{-\sigma t}$. Thus we find $\mu_1^*>0$. Similarly we have $\mu_i^*<0$ for $i\geq 2$. Consequently, we conclude that $-\mu_i^*=\lambda_i^*$, for $i\geq 1$ and thus $u^*$ is an index-1 saddle point, leading to (b). 

If (b) holds, then we can immediately verify that $\mathcal L_{v_1^*}\big(\Delta u^* +f(u^*)\big)=0$ and $ \mathcal N_{v_1^*}\big(\Delta v_1^*+f'(u^*)v_1^*\big)=0$ by $(-\Delta -f'(u^*))v_1^*=\lambda_1^* v_1^*$ and $\|v_1^*\|=1$. Thus $(u^*, v_1^* )$ is a stationary state of (\ref{1order}). By similar derivations around (\ref{for}), we have $\delta_u=\sum_{i=1}^\infty \delta_{u,i}(t)v_i^*(x)$ with $\delta_{u,1}(t)=(\delta _u(0),v_1^*)e^{\beta \lambda_1^*t}$ and $\delta_{u,i}(t)=(\delta _u(0),v_i^*)e^{-\beta \lambda_i^*t}$  for $i\geq 2$. As $u^*$ is an index-1 saddle point, we have $\lambda_1^*<0$ and $\lambda_i^*>0$ for $i\geq 2$ such that
$\|\delta_u\|^2=\sum_{i=1}^\infty \|\delta_{u,i}\|^2\leq e^{2\beta\max\{\lambda_1^*,-\lambda_2^*\}t}\|\delta_u(0)\|^2, $
which implies (a).
\end{proof}

\begin{remark}
In Theorem \ref{thm21}, we have assumed that $u^*,v_1^*\in H^2(\Omega)$ since the semilinear elliptic equation (\ref{elliptic}), the parabolic system (\ref{1order}) and the eigenvalue problem of the operator $-\Delta-f'(u^*)$ are all considered under the $L^2$ sense such that  $u^*,v_1^*\in H^2(\Omega)$ is the minimal requirement to match the $L^2$ setting. 

To relax the regularity condition, we may consider problems under the weak formulation, which could reduce the requirement to $u^*\in H^1_0(\Omega)\cap L^\infty(\Omega)$ and $v_1^*\in H^1_0(\Omega)$. Specifically, if we consider the weak formulation of (\ref{elliptic}), i.e. $-(\nabla u,\nabla \chi)+(f(u),\chi)=0$ for any $\chi\in H^1_0(\Omega)$, then it suffices to consider its solution in $ H^1_0(\Omega)$ and thus assume $u^*\in H^1_0(\Omega)$. Then if
we intend to define the eigenvalue problem by $(\nabla v,\nabla \chi)-(f'(u^*)v,\chi)=\lambda(v,\chi)$ for any $\chi\in H^1_0(\Omega)$, the assumption $u^*\in L^\infty(\Omega)$ could imply $f'(u^*)\in L^\infty(\Omega)$ (if $f'$ is locally Lipschitz continuous) such that the weak formulation is well-defined and admits eigenfunctions in $H^1_0(\Omega)$. Thus it suffices to assume $v_1^*\in H^1_0(\Omega)$. Then we could consider the weak formulation of (\ref{1order}) for any $\chi_1,\chi_2\in H^1_0(\Omega)$
\begin{equation*}
\begin{array}{l}
   (u_t,\chi_1)=\beta\big[-(\nabla u,\nabla\chi_1) +(f(u),v)- 2(v,\chi_1)(-(\nabla u,\nabla v) +(f(u),v))\big],\\[0.1in] 
   (v_t,\chi_2)=\gamma\big[-(\nabla v,\nabla\chi_2)+(f'(u)v,\chi_2)- (v,\chi_2)(-(\nabla v,\nabla v)+(f'(u)v,v))\big],
   \end{array}
\end{equation*}
and re-prove Theorem \ref{thm21} under the weaker setting by the same procedure,  based on $u^*\in H^1_0(\Omega)\cap L^\infty(\Omega)$ and $v_1^*\in H^1_0(\Omega)$.
\end{remark}

 In this work, we make the following assumption on the nonlinear term $f$:

\textit{Assumption A}: $f$ and $f'$ are local Lipschitz functions on $\mathbb R$, that is, for any $r>0$ there exists a constant $L_r>0$ depending on $r$ such that 
\begin{align}\label{lip}
|f(z_1)-f(z_2)|+|f'(z_1)-f'(z_2)|\leq L_r|z_1-z_2|,~~\forall -r\leq z_1, z_2\leq r.
\end{align}
Note that (\ref{lip}) implies the boundedness of $f(z)$ and $f'(z)$ for $-r\leq z\leq r$ with the bound depending on $r$ and $L_r$.
Furthermore, we use $Q$ to denote a generic positive constant that may assume different values at different occurrences, and use $M$, $\tilde Q$, $C_i$ and $Q_i$ with $0\leq i\in\mathbb N$ to denote fixed constants. All these constants may depend on $T$ and certain norms of solutions but are independent from parameters of numerical methods such as the time step size, the number of steps and the spatial mesh size. We will also omit the variable $x$ in functions, e.g. we denote $u(x,t)$ by $u(t)$.

In the rest of the work, we consider (\ref{1order}) on a finite interval $[0,T]$ since it is shown in Corollary \ref{cor1} and Remark \ref{rem4} that the numerical solution at $T$ could approximate the stationary solution $u^*(x)$ with an arbitrarily small error for $T$ large enough.

\section{Analysis of semi-discrete scheme}
\subsection{Numerical scheme}
Define a quasi-uniform partition of $\Omega$ with mesh size $h$, and let $S_h$
be the space of continuous and piecewise linear functions on $\Omega$ with respect to the
partition. The elliptic projection operator $P:  H^1_0\rightarrow S_{h}$ defined by
$\big(\nabla(g- P g), \nabla \chi \big)=0$ for any $\chi\in S_h$
  satisfies the following estimate \cite{Tho}
\begin{equation}\label{eta1}
	\|g- P g\|+h\|\nabla(g-P g)\|\leq Qh^2\|g\|_{H^2} \text{ for any } g\in H^2\cap H^1_0. 
\end{equation}
Furthermore, we have the inverse estimate \cite[Theorem 4.5.11]{Bre}
\begin{align}\label{inv}
\|g_h\|_{L^\infty}\leq Qh^{-\frac{d}{2}}\|g_h\|,~~\text{ for any }g_h\in S_h,
\end{align}
and the following interpolation error estimates hold for interpolation operator $I_h$ \cite[Theorem 4.4.20]{Bre}
\begin{align}\label{inter}
\|g-I_hg\|\leq Qh^2\|g\|_{H^2},~~\|g-I_hg\|_{L^\infty}\leq Qh^{2-\frac{d}{2}}\|g\|_{H^2},\text{ for any } g\in H^2.
\end{align}
To obtain the weak formulation of the problem, we compute the inner product of the equations in (\ref{1order}) with test functions $\chi_1, \chi_2 \in H^1_0$, respectively, to get
\begin{equation}\label{weak}
\left\{\begin{aligned}
&\beta^{-1}(u_{t}, \chi_1) \!+\! (\nabla u, \nabla \chi_1) \!-\!2(\nabla v,  \nabla u)(v,\chi_1) \!=\!\beta^{-1}(\mathcal L_{v}(f(u)),\chi_1), \\[0.05in]
&\gamma^{-1}(v_{t},\chi_2)\!+\!(\nabla v, \nabla \chi_2)\!-\!(\nabla v,\nabla v)(v,\chi_2) \!=\!\gamma^{-1}(\mathcal N_{v}(f'(u)v),\chi_2).
\end{aligned}
    \right.
\end{equation}
Then the semi-discrete approximation of (\ref{1order}) could be expressed as: find $u_h, v_h\in S_h$ such that the following relations hold for any $\chi_1,\chi_2 \in S_h$
\begin{equation}\label{approximation}
\left\{\begin{aligned}
&\beta^{-1}(u_{h,t}, \chi_1) +(\nabla u_h, \nabla \chi_1) -2(\nabla v_h,  \nabla u_h)(v_h,\chi_1)=\beta^{-1}(\mathcal L_{v_h}(f(u_h)),\chi_1),\\
%&\qquad =(f(u_h), \chi_1)-2\big(v_h, f(u_h)\big)(v_h,\chi_1),\\[0.05in]
&\gamma^{-1}(v_{h,t},\chi_2)+(\nabla v_h, \nabla \chi_2)-(\nabla v_h,\nabla v_h)(v_h,\chi_2)=\gamma^{-1}(\mathcal N_{v_h}(f'(u_h)v_h),\chi_2),
%&\qquad =\big(f'(u_h)v_h,\chi_2\big)-\big(f'(u_h)v_h, v_h\big)(v_h,\chi_2),
\end{aligned}
    \right.
\end{equation}
with $u_h(x,0)=u_{0,h}:=P u_0$ and $v_h(x,0)=v_{0,h}:=\frac{P v_0}{\|P v_0\|} $ such that $\|v_{0,h}\|=1$. By a similar derivation as for (\ref{1order}), $\|v_h\|=1$ for any $t\in [0,T]$. Now we show that $v_{0,h}$ is well approximation to $v_0\in H^2(\Omega)$. By (\ref{eta1}) and $v_{0,h}-Pv_0=\frac{P v_0}{\|P v_0\|}(1-\|Pv_0\|)$ we have 
$$\|v_{0,h}-v_0\|\leq \|v_{0,h}-P v_0\|+\|Pv_0-v_0\|= |\|P v_0\|-1|+\|Pv_0-v_0\|\leq 2\|P v_0 - v_0\|\leq Qh^2.$$ Furthermore, $\|v_0-Pv_0\|\leq Qh^2$ implies that $\|Pv_0\|\geq 1-Qh^2\geq 1/2$ for $h$ small enough, which in turn leads to 
\begin{align}\label{gv0}
\|\nabla v_{0,h}\|=\frac{\|\nabla Pv_0\|}{\|Pv_0\|}\leq 2\|\nabla Pv_0\|\leq 2\|\nabla v_0\|.
\end{align}

\subsection{Well-posedness issue} We derive stability estimates and uniqueness of the numerical solutions to the semi-discrete scheme (\ref{approximation}) under the condition below:

 \textit{Assumption B}:  There exists an $h_0>0$ such that the solution $u_h$ of (\ref{approximation}), if exists, satisfies 
$	\|u_h\|_{L^{\infty}(0,T; L^{\infty})}\leq C_0$ for any $0<h<h_0$ and for some fixed $C_0>0$. 
\begin{remark}
The Assumption B imposes the boundedness of the numerical solutions. Note that similar boundedness assumptions are often imposed in numerical studies of nonlinear PDEs, e.g. the three-dimensional Navier-Stokes equation \cite{Li}.
 
 If we consider proving the bound of $u_h$, we may first recall the semilinear parabolic equation $
 u_t=\Delta u+f(u)$, a similar but much simpler equation in comparison to the coupled parabolic system (\ref{1order}). For such equation, the boundedness of its semi-discrete numerical solutions is usually derived from that of the true solutions based on the error estimate,  see e.g. \cite{Lar}. Specifically, one first introduce an auxiliary semi-discrete scheme with a cutoff nonlinearity $\tilde f(\cdot)$, which equals to $f(\cdot)$ over some bounded interval $\tilde I$ (e.g. $\tilde I:=[-\|u\|_{L^\infty}-1,\|u\|_{L^\infty}+1]$) but is globally Lipschitz continuous such that the $L^2$ error estimate for numerical solutions $\tilde u_h$ of the auxiliary scheme can be derived as $\|\tilde u-u_h\|\leq Qh^2$. 
 Then we combine this with the interpolation estimates (\ref{inter}) and inverse estimate (\ref{inv}) to bound $\tilde u_h$ by
 \begin{align}
 \|\tilde u_h\|_{L^\infty}&\leq  \|\tilde u_h-I_hu\|_{L^\infty}+ \|I_hu-u\|_{L^\infty}+ \|u\|_{L^\infty}\nonumber\\
&\hspace{-0.2in} \leq Qh^{-\frac{d}{2}}\|\tilde u_h-I_hu\|+Qh^{2-\frac{d}{2}}\|u\|_{H^2}+\|u\|_{L^\infty}\label{linf}\\
& \hspace{-0.2in}\leq Qh^{-\frac{d}{2}}(\|\tilde u_h-u\|+\|u-I_hu\|)+Qh^{2-\frac{d}{2}}\|u\|_{H^2}+\|u\|_{L^\infty}\leq \|u\|_{L^\infty}+Qh^{2-\frac{d}{2}}.\nonumber
 \end{align}
As $2-\frac{d}{2}>0$ for $1\leq d\leq 3$,  $\tilde u_h\in\tilde I$ for $h $ small enough such that the auxiliary scheme is consistent with the original scheme and we immediately get $\tilde u_h=u_h$ and $ \|u_h\|_{L^\infty(0,T;L^\infty)}\leq  \|u\|_{L^\infty(0,T;L^\infty)}+1$ (i.e. the Assumption B).
 
 For the coupled parabolic system (\ref{1order}), however, the above idea could not be applied for its semi-discrete scheme (\ref{approximation}). Due to the gradient nonlinearity in the dynamics of $v$ and its coupling with the dynamics of $u$, the cutoff function in the auxiliary scheme should also include the gradient as a variable, and we thus need to estimate the error of the gradient and then accordingly bound the gradient to recover the original scheme. Assume we could ideally obtain $\|\nabla(\tilde v_h-v)\|\leq Qh$ ($\tilde v_h$ denotes the numerical solution of the auxiliary semi-discrete scheme for the dynamics of $v$), which is true and optimal for the semi-discrete scheme of linear parabolic equations, see e.g. \cite[Theorem 1.3]{Tho}.  Then a similar estimate as (\ref{linf}) for $\nabla \tilde v_h$ will generate an $O(h^{1-\frac{d}{2}})$ term, which is not an infinitesimal for $d=2,3$. This may lead to the failure of the recovery of the original scheme at least for $d=2,3$. 
 
 Based on the above discussions and attempts, it remains a  challenge to bound $u_h$ for the semi-discrete scheme (\ref{approximation}) due to the gradient nonlinearity and we thus impose the Assumption B for the sake of numerical analysis. Nevertheless, we will investigate how to relax or even prove the Assumption B in the future. 
\end{remark}

\begin{lemma}\label{lem0}
Under the Assumptions A--B, the numerical solutions of (\ref{approximation}) satisfy
$
	\| \nabla  u_h\|+\|\nabla v_h\| \leq Q
$ on $t\in [0,T]$ for $h<h_0$ given in Assumption B.
\end{lemma}

\begin{proof}
By {\it Assumption B} and (\ref{lip}), we have $|f'(u_h)|\leq Q$. We invoke this in the second equation of (\ref{approximation}) with $\chi_2=v_{h, t}$ and apply $\|v_h\|=1$ (such that $(v_h,v_{h,t})=0$)
and Young's inequality to get
\begin{equation*}
\displaystyle \gamma^{-1}\| v_{h,t}\|^2+\frac{1}{2}\frac{d}{dt}\|\nabla v_{h}\|^2  = (f'(u_h)v_h, v_{h, t}) \leq Q\| v_{h,t}\| \leq Q+\gamma^{-1}\|v_{h, t}\|^2,
\end{equation*}
which implies $\displaystyle \frac{d}{dt}\|\nabla v_{h}\| \leq Q$. We integrate this inequality from $0$ to $t$ and apply (\ref{gv0}) to obtain
$\|\nabla  v_{h}\| \leq \|\nabla  v_{0,h} \|+Q \leq Q(\|\nabla  v_0\|+1).$
We then apply this estimate,  $\|v_h\|=1$, and {\it Assumption B} in the first equation of (\ref{approximation}) with $\chi_1= u_{h,t}$ to get
\begin{equation*}
\begin{aligned}
\beta^{-1}\|u_{h, t}\|^2+\frac{d \|\nabla u_h\|^2 }{2dt}& =\!2(\nabla u_h,\!\nabla v_h)(v_h, u_{h, t})\!+ (f(u_h), u_{h, t})\!-2(v_h, f(u_h))(v_h, u_{h, t}) \\
& \leq Q\|\nabla  u_h\| \cdot\|u_{h, t}\|+Q\|u_{h, t}\| \leq Q\|\nabla u_h\|^2+Q+\beta^{-1}\|u_{h, t}\|^2,
	\end{aligned}
\end{equation*}
that is, 
$
\displaystyle  \frac{d}{dt}\|\nabla u_h\|^2  \leq  Q+Q\|\nabla u_h\|^2
$. An application of the Gronwall inequality gives
$
	\| \nabla u_h\|\leq Q+Q\| \nabla u_{0,h}\|\leq Q(1+\|\nabla u_0\|).
$
Thus the proof is completed.
\end{proof}

\begin{theorem}
Under the Assumptions A--B, the solution of  (\ref{approximation}) is unique for $h<h_0$ given in Assumption B.
\end{theorem}

\begin{proof}
Suppose the scheme (\ref{approximation}) admits two pairs of solutions $(u_h,v_h)$ and $(\bar u_h,\bar v_h)$ where $u_h$ and $\bar u_h$ satisfy {\it Assumption B}.  Let
 $\mu_h=u_h-\bar u_h$ and $\nu_h=v_h-\bar v_h$, which satisfy $\mu_h(0)=\nu_h(0)=0$.
By subtracting the equations of $u_h$ and $\bar u_h$, we obtain 
\begin{equation}\label{epu}
	\begin{aligned}
&\beta ^{-1}(\mu_{h, t}, \chi_1) +(\nabla \mu_h, \nabla  \chi_1)-2\big((\nabla u_h,\nabla  v_h)(v_h, \chi_1)-(\nabla \bar u_h,\nabla \bar v_h)(\bar v_h, \chi_1)\big)  \\[0.05in]
&\qquad =\big(f(u_h)- f(\bar u_h),\chi_1\big)-2\big((f(u_h), v_h)(v_h, \chi_1)-(f(\bar u_h), \bar v_h)(\bar v_h, \chi_1 )\big).
	\end{aligned}
\end{equation}
By setting $\chi_1= \mu_h$ in (\ref{epu}) and applying Lemma \ref{lem0}, we derive 
\begin{equation}
\begin{aligned}\label{epua}
&\big|(\nabla u_h,\nabla  v_h)(v_h,  \mu_h)-(\nabla \bar u_h,\nabla \bar v_h)(\bar v_h,  \mu_h)\big|=\big|(\nabla u_h -\nabla \bar u_h , \nabla  v_h)(v_h,  \mu_h)\\[0.05in]
&\qquad +(\nabla \bar u_h,\nabla v_h-\nabla \bar v_h)(v_h,  \mu_h)+(\nabla \bar u_h,\nabla \bar v_h)(v_h -\bar v_h, \mu_h )\big |\\[0.05in]
& \quad \leq Q\|\nabla  \mu_h\|\cdot \| \mu_h\|+Q\|\nabla \nu_h\|\cdot \| \mu_h\|+Q\|\nu_h\|\cdot \| \mu_h\|.
\end{aligned}
\end{equation}
Similarly, by invoking \textit{Assumption A}, we have
\begin{equation}
	\begin{aligned}\label{epub}
&\big|(f(u_h), v_h)(v_h,  \mu_h)-(f(\bar u_h), \bar v_h)(\bar v_h,  \mu_h)\big| \\[0.05in]
&\qquad=\big |(f(u_h)-f(\bar u_h), v_h)(v_h,  \mu_h) +(f(\bar u_h), v_h-\bar v_h)(v_h,  \mu_h)\\[0.05in]
&\quad\qquad +(f(\bar u_h), \bar v_h)(v_h-\bar v_h,  \mu_h)\big|\leq Q\| \mu_h\|^2+Q\| \nu_h\|\cdot \| \mu_h\|.
	\end{aligned}
\end{equation}
Invoke the above two estimates in (\ref{epu}) to get
\begin{equation}\label{epu1}
\begin{aligned}
\displaystyle \frac{d}{dt}\| \mu_h\|^2+2\beta \|\nabla  \mu_h\|^2 & \leq 
Q\big[\|\nabla  \mu_h\|\cdot \| \mu_h\|+\|\nabla \nu_h\| \cdot \| \mu_h\| +\| \nu_h\|\cdot \|\mu_h\| + \|\mu_h\|^2\big],
\end{aligned}
\end{equation}
which, together with the Young’s inequality, leads to
$\frac{d}{dt}\| \mu_h\|^2 \leq Q\| \mu_h\|^2+Q\| \nu_h\|^2+Q\|\nabla \nu_h\|^2. $
Then we set $\chi_1=\mu_{h, t}$ in (\ref{epu}) and apply a similar process as above to get 
\begin{equation}\label{epu2}
\begin{aligned}
\frac{d}{dt}\|\nabla \mu_h\|^2 \leq Q \| \mu_h\|^2+Q\| \nu_h\|^2+Q\|\nabla \nu_h\|^2+Q\|\nabla \mu_h\|^2.
\end{aligned}
\end{equation}
To estimate $\| \nu_h\|$, we subtract schemes of $v_h$ and $\bar v_h$ to get 
\begin{align}\label{epv}
& \gamma^{-1}(\nu_{h,t}, \chi_2) +(\nabla  \nu_h, \nabla \chi_2)-\big((\nabla v_h,\nabla v_h)(v_h,\chi_2)-(\nabla \bar v_h,\nabla \bar v_h)(\bar v_h,\chi_2)\big)\nonumber\\[0.05in]
&\quad =\big((f'(\bar u_h)\bar v_h, \bar v_h)(\bar v_h,\chi_2)-(f'(u_h)v_h, v_h)(v_h,\chi_2)\big)\\[0.05in]
&\qquad +(f'(u_h)v_h-f'(\bar u_h)\bar v_h,\chi_2).\nonumber
\end{align}
Setting $\chi_2= \nu_h$ in (\ref{epv}) and applying similar estimates as (\ref{epua}), we have 
\begin{equation}
	\begin{aligned}\label{epva}
\big|(\nabla v_h,\nabla v_h)(v_h,\nu_h)-(\nabla \bar v_h,\nabla \bar v_h)(\bar v_h,\nu_h)\big|
%&\quad=\big|(\nabla v_h,\nabla v_h)(v_h,\chi_2)-(\nabla \bar v_h,\nabla v_h)(v_h,\chi_2)\\[0.05in]
%&\qquad +(\nabla \bar v_h,\nabla v_h)(v_h,\chi_2)-(\nabla \bar v_h,\nabla \bar v_h)(v_h,\chi_2)\\[0.05in]
%&\qquad+(\nabla \bar v_h,\nabla \bar v_h)(v_h,\chi_2)-(\nabla \bar v_h,\nabla \bar v_h)(\bar v_h,\chi_2)\big|\\[0.05in]
\leq Q\|\nabla \nu_h\|\cdot \|\nu_h\|+ Q\|\nu_h\|^2.
\end{aligned}
\end{equation}
By the Sobolev embedding $H^{1}(\Omega)\hookrightarrow L^4(\Omega)$, the Poincar\'e inequality and Lemma \ref{lem0}, we derive
\begin{align}
&\label{epvb}\big|(f'(u_h)v_h-f'(\bar u_h)\bar v_h,\nu_h)\big|=\big|\big( (f'(u_h)-f'(\bar u_h))v_h + f'(\bar u_h)(v_h-\bar v_h),\nu_h\big) \big|\\[0.05in]
& ~ \leq Q(|\mu_h|\cdot |v_h|,|\nu_h|)+ Q\|\nu_h\|^2 \leq Q\|\mu_h v_h\|\cdot \|\nu_h\|+ Q\|\nu_h\|^2\nonumber \\[0.05in]
&~  \leq Q\|\mu_h\|_{L^4}\cdot \|v_h\|_{L^4}\cdot \|\nu_h\|+Q\|\nu_h\|^2  \leq Q\|\mu_h\|_{H^1}\cdot \|v_h\|_{H^1}\cdot \|\nu_h\|+Q\|\nu_h\|^2\nonumber\\[0.05in]
&~  \leq Q\|\nabla \mu_h\|\cdot \|\nu_h\|+ Q\|\nu_h\|^2.\nonumber
\end{align}
Combine a similar derivation as (\ref{epvb}) and  $H^{1}(\Omega)\hookrightarrow L^4(\Omega)$ to get
\begin{equation}
	\begin{aligned}\label{epvc}
& \big|(f'(\bar u_h)\bar v_h, \bar v_h)(\bar v_h,\nu_h)-(f'(u_h)v_h, v_h)(v_h,\nu_h)\big|\\[0.05in]
%&\quad=\big|(f'(\bar u_h)\bar v_h, \bar v_h)(\bar v_h,\chi_2)-(f'(u_h)\bar v_h, \bar v_h)(\bar v_h,\chi_2)\\[0.05in]
%&\qquad +(f'(u_h)\bar v_h, \bar v_h)(\bar v_h,\chi_2)-(f'(u_h)v_h, \bar v_h)(\bar v_h,\chi_2)\\[0.05in]
%&\qquad +(f'(u_h)v_h, \bar v_h)(\bar v_h,\chi_2)-(f'(u_h)v_h, v_h)(\bar v_h,\chi_2)\\[0.05in]
%&\qquad +(f'(u_h)v_h, v_h)(\bar v_h,\chi_2)-(f'(u_h)v_h, v_h)(v_h,\chi_2)\big|\\[0.05in]
&\qquad\leq Q\|\mu_h\|\cdot\|\bar v_h\|^2_{L^4} \cdot \|\nu_h\|+Q\|\nu_h\|^2\leq Q\|\mu_h\|\cdot\|\nu_h\|+Q\|\nu_h\|^2.
	\end{aligned}
\end{equation}
Substituting the above estimates in (\ref{epv}) and using the Young’s inequality we get
\begin{equation}\label{epv1}
	\begin{aligned}
\frac{d}{dt}\|\nu_h\|^2 \leq  Q\|\nu_h\|^2+Q\|\mu_h\|^2+Q\|\nabla \mu_h\|^2.
	\end{aligned}
\end{equation}
We finally set $\chi_2=\nu_{h,t}$ in (\ref{epv}) and apply a similar process as  (\ref{epva})--(\ref{epvc}) to get 
\begin{equation}\label{epv2}
	\begin{aligned}
\frac{d}{dt}\|\nabla \nu_h\|^2  \leq Q\|\nu_h\|^2+ Q\|\nabla \nu_h\|^2+Q\|\mu_h\|^2+Q\|\nabla \mu_h\|^2.
	\end{aligned}
\end{equation}
Summing (\ref{epu1}), (\ref{epu2}), (\ref{epv1}) and (\ref{epv2}) together and invoking the Gronwall inequality, we conclude that $\|\mu_h\|+\|\nabla \mu_h\|+\|\nu_h\|+\|\nabla \nu_h\|=0$, which completes the proof.
\end{proof}

\subsection{Error estimates}\textbf{}
We prove error estimates for (\ref{approximation}).
\begin{theorem}\label{thm33}
Suppose the Assumptions A--B hold, and $u,v\in W^{1,1}(0,T;H^2)$. Then the  error estimate $\| u- u_h\|+\| v- v_h\|\leq Qh^2$ for the semi-discrete scheme (\ref{approximation}) holds for $0\leq t\leq T$ and for $h<h_0$ given in Assumption B.
\end{theorem}

\begin{proof}
We decompose $ u- u_h= u-P u+ P u-u_h=: \eta_u+\xi_u$ and $v- v_h= v-P  v+P v - v_h=:\eta_v+\xi_v$.  Subtracting the first equation of  (\ref{approximation}) from that of (\ref{weak}) and setting $\chi_1=\xi_u$ in the resulting equation, we derive the error equation for $u$
\begin{equation}
	\begin{aligned}\label{errorU}
& \beta^{-1}(\xi_{u,t}+\eta_{u,t},\xi_u)+(\nabla \xi_u, \nabla \xi_u)-2\big((\nabla  u,\nabla  v)(v, \xi_u)-(\nabla u_h,\nabla v_h)(v_h, \xi_u)\big) \\[0.05in]
&\qquad =\big(f(u)- f(u_h),\xi_u\big)-2\big(( f(u), v)(v,\xi_u)-(f(u_h), v_h)(v_h,\xi_u)\big).
\end{aligned}
\end{equation}
By Lemma \ref{lem0}, $\big(\nabla u,\nabla(v-v_h)\big)=(-\Delta u, v-v_h)$ and $W^{1,1}(0,T;H^2) \hookrightarrow C([0,T];H^2)$, we have
\begin{equation*}
\begin{aligned}
&2\big |(\nabla  u_h,\nabla  v_h)( v_h, \xi_u) -(\nabla u,\nabla  v)( v, \xi_u)\big|= 2\big |(\nabla (u_h-u),\nabla  v_h)( v_h, \xi_u)\\[0.05in]
& \qquad + (\nabla u,\nabla(v_h-v))(v_h, \xi_u)+ (\nabla u,\nabla v)(v_h-v, \xi_u)\big| \\[0.05in]
& \quad \leq \|\nabla \xi_u\|^2+Q\|\xi_u\|^2+2\|\Delta  u\|\cdot \| v- v_h\|\cdot\|\xi_u\|+Q\|v-v_h\|\cdot\|\xi_u\|.
\end{aligned}
\end{equation*}
 Similar to (\ref{epub}), we apply $W^{1,1}(0,T;H^2)  \hookrightarrow C([0,T];L^\infty)$ to get
\begin{equation*}
	|(f(u), v)( v, \xi_u)-(f(u_h), v_h)( v_h,\xi_u)| \leq Q\| v- v_h\|\cdot \|\xi_u\|+Q\|u-u_h\|\cdot \|\xi_u\|.
\end{equation*}
Substituting these estimates into (\ref{errorU}) and using $\|\eta_{u,t}\|\leq Q\|u_t\|_{H^2}h^2$ according to (\ref{eta1}), we obtain
\begin{equation*}
	\begin{aligned}
\frac{\beta^{-1}}{2}\frac{d}{dt}\|\xi_u\|^2 +\|\nabla \xi_u \|^2 & \leq \|\nabla \xi_u\|^2+Q\|\xi_u\|^2+2\|\Delta  u\|\cdot \| v- v_h\|\cdot\|\xi_u\|\\
&\quad  +Q\|v-v_h\|\cdot\|\xi_u\|+Q\|u-u_h\|\cdot \|\xi_u\|+Qh^2\|u_t\|_{H^2}\cdot\|\xi_u\|  ,
\end{aligned}
\end{equation*}
for which we could apply $\frac{d}{dt}\|\xi_u\|^2=2\|\xi_u\|\frac{d}{dt}\|\xi_u\|$, cancel $\|\xi_u\|$ on both sides and employ $\|u-u_h\|\leq \|\xi_u\|+\|\eta_u\|\leq \|\xi_u\|+Q\|u\|_{H^2}h^2$ to get
$$\beta^{-1} \frac{d}{dt}\|\xi_u\|   \leq Q\big[\| v- v_h\|+h^2\|u\|_{H^2}+\|\xi_u\|+h^2\|u_t\|_{H^2}\big]. $$
An application of the Gronwall inequality leads to
\begin{equation}\label{xiU}
	\|\xi_u\|\leq Q\int_0^t\| v - v_h\|ds+Q\|u\|_{W^{1,1}(0,T;H^2)}h^2\leq Q\int_0^t\| \xi_v\|ds+Qh^2.
\end{equation}

To estimate the error of $v-v_h$, we subtract the second equation of (\ref{approximation}) from that of (\ref{weak}) and set $\chi_2=\xi_v$ to get
\begin{equation}
\begin{aligned}\label{v}
&\gamma^{-1}(\xi_{v,t}+\eta_{v,t},\xi_v)+(\nabla \xi_v, \nabla \xi_v)-\big((v,\xi_v)(\nabla  v,\nabla v)-(v_h,\xi_v)(\nabla v_h,\nabla v_h)\big)\\[0.05in]
&\quad=(f'(u) v- f'(u_h)v_h,\xi_v)-\big ((v,\xi_v)( f'(u) v, v)-(v_h,\xi_v)(f'(u_h)v_h, v_h)\big).
\end{aligned}
\end{equation}
We apply a similar splitting argument as we did before and $(\nabla (v-v_h), \nabla v)= -(v-v_h, \Delta v)$ to get
\begin{equation*}
	\begin{aligned}
&\big|( v,\xi_v)(\nabla  v,\nabla  v)-(v_h,\xi_v)(\nabla v_h,\nabla v_h) \big|\!=\!\big|(v, \xi_v)\big[(\nabla (v-v_h),\nabla v)+ (\nabla v_h,\nabla (v-v_h))\big] \\[0.05in]
& ~ + (v-v_h, \xi_v)(\nabla v_h,\nabla v_h)\big| \leq Q\|\xi_v\|\| v-v_h\|+\|\nabla \xi_v\|^2+Q\|\xi_v\|^2 .
	\end{aligned}
\end{equation*}
By employing $W^{1,1}(0,T;H^2)  \hookrightarrow C([0,T];L^\infty)$ we derive
\begin{equation*}
	\begin{aligned}
 & |(f'(u) v- f'(u_h) v_h,\xi_v)|=|(f'(u_h)(v-v_h)+(f'(u)-f'(u_h))v, \xi_v )| \\[0.05in]
 &\qquad \leq \big(\|f'(u_h)\|_{L^{\infty}} \|v-v_h\|+\|v\|_{L^{\infty}} \|f'(u)-f'(u_h)\| \big)\cdot \|\xi_v\|\\[0.05in]
 &\qquad \leq Q\big(\| u- u_h\|+\| v- v_h\|\big) \|\xi_v\|  ,\\[0.05in]
 &|(v,\xi_v)( f'(u) v, v)-( v_h,\xi_v)( f'(u_h) v_h, v_h)| \leq Q\big(\| v- v_h\|+\| u-u_h\|\big) \|\xi_v\|.
\end{aligned}
\end{equation*}
We substitute the above estimates in (\ref{v}) and apply $\frac{d}{dt}\|\xi_v\|^2=2\|\xi_v\|\frac{d}{dt}\|\xi_v\|$  to get
\begin{align}
\frac{d}{dt}\|\xi_v\|\leq Q\big[\|\xi_v\|+\|u-u_h\|+\|v-v_h\|\big]\leq Q\big[\|\xi_v\|+\|\xi_u\|+\|u\|_{H^2}h^2+\|v\|_{H^2}h^2\big]. \nonumber
\end{align}
An application of the Gronwall inequality gives
$
	\|\xi_v\| \leq Qh^2+\int_0^t\|\xi_u\|ds.
$
We add this and  (\ref{xiU}) and apply the Gronwall inequality again to get
$
	\|\xi_u\|+\|\xi_v\|\leq Qh^2,
$
which, together with (\ref{eta1}), completes the proof.
\end{proof}

Based on the error estimate, we could give a theoretical support for the index-preservation issue of the scheme.
\begin{corollary}\label{cor1}
Under the conditions of Theorem \ref{thm33}, if $\lim\limits_{t\rightarrow \infty}u(t)=u^*$ under the $L^2$ sense for an index-1 saddle point $u^*$, then for any $\delta>0$, the approximation $\|u_h(T)-u^*\|\leq \delta$ holds for $T$ large enough and $h$ small enough.

In addition, if for any $\hat u\in B_\delta(u^*):=\{\hat u\in L^2:\,\|\hat u-u^*\|\leq \delta\}$ for some $\delta>0$ the eigenvalues of $-\Delta-f'(\hat u)$ and $-\Delta-f'( u^*)$ under zero Dirichlet boundary conditions have the same signs, then the eigenvalues of  $-\Delta-f'(u_h(T))$ and $-\Delta-f'( u^*)$ also have the same signs for $T$ large enough and $h$ small enough, that is, the scheme (\ref{approximation}) is index-preserving. 
\end{corollary}
\begin{proof}
As $\lim\limits_{t\rightarrow \infty}u(t)=u^*$ under the $L^2$ sense, there exist a $T_0>0$ such that $\|u(T_0)-u^*\|\leq \delta/2$.  According to Theorem \ref{thm33} with $T=T_0$, for $h$ small enough we have $\|u_h(T_0)-u(T_0)\|\leq QM_0h^2\leq \delta/2$. Thus we have $\|u_h(T)-u^*\|\leq \delta$. The second statement of the theorem is a direct consequence of this estimate.
\end{proof}
\begin{remark}
Corollary \ref{cor1} implies that $u_h(T)$ could approximate $u^*$ with an arbitrarily small error for $T$ sufficiently large and $h$ small enough, which justifies the sufficiency of considering (\ref{1order}) on a finite interval $[0,T]$.
Furthermore, the Corollary \ref{cor1} could also be rewritten under the $L^\infty$ sense. In this case, we need to estimate $\|u-u_h\|_{L^\infty}$, which could be obtained from the $L^2$ error estimate in Theorem \ref{thm33}, the interpolation estimate (\ref{inter}) and the inverse estimate (\ref{inv})
\begin{equation*}
	\begin{aligned}
	\displaystyle  \|u-u_h\|_{L^\infty} &\leq \| u-I_h u\|_{L^\infty}+\|I_h u- u_h\|_{L^\infty}\leq Qh\|u\|_{H^2}+Qh^{-\frac{d}{2}}\|I_h u- u_h\|\\[0.05in]
	\displaystyle  &\hspace{-0.5in}\leq Qh+ Qh^{-\frac{d}{2}}\big(\|I_h u- u\|+\| u- u_h\|\big)\leq Qh+Qh^{-\frac{d}{2}}(h^2+h^2)\leq Qh^{\min\{1,2-\frac{d}{2}\}}.
	\end{aligned}
\end{equation*}
\end{remark}
\begin{remark}\label{rem33}
The Corollary \ref{cor1} depends on the following assumption ``for any $\hat u\in B_\delta(u^*)$ for some $\delta>0$ the eigenvalues of $-\Delta-f'(\hat u)$ and $-\Delta-f'( u^*)$ under zero Dirichlet boundary conditions have the same signs''. In this remark we will prove that this assumption is valid at least for a wide class of one-dimensional  problems. 

Consider the following eigenvalue problem
\begin{align}\label{zjy10}
(-\Delta +q(x))u(x)=\lambda u(x),~~x\in (0,1);~~u(0)=u(1)=0,
\end{align}
which corresponds to the model (1.1) proposed in \cite{LiQi} with $\mu=0$. Here $q\in L^2$ is a potential function. It is shown in \cite[Proposition 2.3]{LiQi} that:
\begin{align*}
&(P_0)\text{ All eigenvalues of (\ref{zjy10}) are simple and satisfy }\\
&~~~\,~~-\infty<\lambda_{1,q}<\lambda_{2,q}<\cdots\lambda_{n,q}<\cdots,~~\lambda_{n,q}\sim\pi^2n^2,~~n\rightarrow \infty.
\end{align*}
Here $\lambda_{n,q}$ denotes the $n$-th eigenvalue of (\ref{zjy10}) with the potential function $q$. 
Then \cite[Theorem 3.2]{LiQi} reveals the continuous dependence of the $n$-th eigenvalue of (\ref{zjy10}) on the potential function $q$: 
\begin{align*}
&(P_1) \text{ Let }n\geq 1 \text{ and }q_0\in L^2 \text{ be fixed. For any }\varepsilon>0 \text{ there exists a }\delta_n>0\\ 
&~~~~~\text{ such that if }\|q_1-q_0\|\leq \delta_n \text{ for any }q_1\in L^2, \text{ then }|\lambda_{n,q_1}-\lambda_{n,q_0}|<\varepsilon.
\end{align*}

Now we intend to use the results $(P_0)$--$(P_1)$ to show that, if $f'$ satisfies the local Lipschitz condition under the $L^2$ sense, i.e.,
\begin{align}\label{zjy14}
\|f'(u_1)-f'(u_2)\|\leq \tilde L_r\|u_1-u_2\|\text{ for some }\tilde L_r>0\text{ if }u_1,u_2\in B_r(0)\text{ for some }r>0,
\end{align}
then the assumption on the signs of eigenvalues in Corollary \ref{cor1} can be justified, i.e., the eigenvalues of $-\Delta-f'(\hat u)$ and $-\Delta-f'(u^*)$ that are defined on $\Omega=(0,1)$ and equipped with zero Dirichlet boundary conditions have the same signs for $\hat u\in B_\delta(u^*)$ for some $\delta>0$ (to be specified later). 

To start the proof, suppose $u^*$ is an index-$1$ saddle point. Since $u^*\in L^2$, (\ref{zjy14}) implies that $f'(u^*)\in L^2$. Thus we apply $(P_0)$ to find that the eigenvalues of  $-\Delta-f'(u^*)$ are simple and can be denoted as  $\{\lambda_{n,-f'(u^*)}\}_{n=1}^\infty$ with $\lambda_{1,-f'(u^*)}<0<\lambda_{2,-f'(u^*)}<\lambda_{3,-f'(u^*)}<\cdots$. Denote $\sigma:=\frac{1}{2}\min\{-\lambda_{1,-f'(u^*)},\lambda_{2,-f'(u^*)}\}>0$, and we apply $(P_1)$ twice with $n=1,2$ to find that for $\sigma>0$, there exists a $\tilde\delta:=\min\{\delta_1,\delta_2\}$ such that $|\lambda_{n,q}-\lambda_{n,-f'(u^*)}|<\sigma$ for $n=1,2$ and $q\in L^2$ satisfying $\|q-(-f'(u^*))\|\leq \tilde\delta$. By the definition of $\sigma$, we immediately obtain $\lambda_{q,1}<0<\lambda_{q,2}$ if $\|q-(-f'(u^*))\|\leq \tilde\delta$. Since the eigenvalues of $-\Delta+q$ are monotonically increasing (cf. $(P_0)$), we conclude that:
\begin{align*}
&(P_2)\text{ The eigenvalues of }-\Delta+q \text{ and }-\Delta-f'(u^*)\\
&~~\,~~~\text{  have the same signs if }\|q-(-f'(u^*))\|\leq \tilde\delta.
\end{align*}
To complete the proof, we remain to show that $\|(-f'(\hat u))-(-f'(u^*))\|\leq \tilde\delta$  for $\|\hat u-u^*\|\leq \delta$ with $\delta:=\min\{1,\tilde\delta /\tilde L_{\|u^*\|+1}\}$. Indeed, if $\|\hat u-u^*\|\leq \delta$, then $\|\hat u-u^*\|\leq 1$ such that $\|\hat u\|\leq \|u^*\|+1$, and we thus employ (\ref{zjy14}) to get
$$\|(-f'(\hat u))-(-f'(u^*))\|\leq \tilde L_{\|u^*\|+1}\|\hat u-u^*\|\leq \tilde L_{\|u^*\|+1}\delta \leq \tilde\delta. $$
Consequently, we apply $(P_2)$ to conclude that the eigenvalues of $-\Delta-f'(\hat u)$ and $-\Delta-f'(u^*)$ have the same signs for $\hat u\in B_\delta(u^*)$ with $\delta:=\min\{1,\tilde\delta /\tilde L_{\|u^*\|+1}\}$, which justifies the assumption on the signs of eigenvalues in Corollary \ref{cor1}.

Finally, we give some examples of $f(u)$ on $\Omega=(0,1)$ that satisfy the Assumption A and (\ref{zjy14}). We could directly verify that $f(u)=a\sin u+b\cos u$ for $a,b\in\mathbb R$, $f(u)=(1+u^2)^{1/2}$ or $f(u)=e^{-u^2}$ satisfy the required conditions. For such $f$, the assumption on the signs of eigenvalues in Corollary \ref{cor1} is valid.
\end{remark}
\section{Analysis of full discretization}
\subsection{Numerical scheme}
 Given $0<N\in \mathbb{N}$, we define a uniform temporal partition of the interval $[0,T]$ by $t_n=n\tau$ for $0\leq n\leq N$ with time step size $\tau = T/N$. 
We approximate the derivative by the backward Euler scheme at $t_n$ as follows
\begin{equation*}
	y_t(t_n)=\bar \partial_t y_n+R_n^y := \frac{y(t_n)-y(t_{n-1})}{\tau}+\frac{1}{\tau}\int _{t_{n-1}}^{t_n} y_{tt}(t)(t-t_{n-1})dt,
\end{equation*}
where $y$ refers to $u$ or $v$. Invoking this approximation into (\ref{weak}) yields
\begin{equation}\label{Euler}
\left\{
\begin{array}{l}
\beta ^{-1}\big(\bar \partial _t u(t_n)+R_n^u, \chi_1 \big) +  \big(\nabla u(t_{n}), \nabla \chi_1 \big) -  2\big(\nabla v(t_{n}), \nabla u(t_{n})\big)(v(t_{n}), \chi_1) \\[0.1in]
 \qquad =  (f(u(t_{n})), \chi_1) - 2\big(f(u(t_{n})), v(t_{n})\big)(v(t_{n}),\chi_1),\\[0.15in]
\gamma^{-1}\big(\bar \partial _t v(t_n)+R_n^v, \chi_2 \big) + (\nabla v(t_{n}), \nabla \chi_2)-(v(t_{n}), \chi_2)\big(\nabla v(t_{n}),\nabla v(t_{n})\big)\\[0.1in]
 \qquad = \big(f'(u(t_{n}))v(t_{n}),\chi_2 \big)- (v(t_{n}),\chi_2)\big(f'(u(t_{n}))v(t_{n}), v(t_{n})\big).
\end{array}
\right.
\end{equation}
We drop the truncation errors to get the fully-discrete scheme: find $\{u_h^n, v_h^n)\}_{n=1}^{N}$ such that for any $\chi_1,\chi_2 \in S_h$
\begin{equation}\label{app}
\left\{
\begin{array}{l}
 \displaystyle \beta ^{-1}\big(\frac{u^n_h-u^{n-1}_h}{\tau}, \chi_1\big)+(\nabla u^{n}_h, \nabla \chi_1)-2(v^{n-1}_h, \chi_1)(\nabla v^{n-1}_h, \nabla u^{n}_h)\\[0.1in]
\qquad  =(f(u^{n-1}_h), \chi_1)-2(v^{n-1}_h,\chi_1)(f(u^{n-1}_h), v^{n-1}_h),\\[0.15in]
\displaystyle \gamma^{-1}\big(\frac{ v^{*,n}_h-v^{n-1}_h}{\tau}, \chi_2\big)+(\nabla v^{*,n}_h, \nabla \chi_2)-(v^{n-1}_h,\chi_2)(\nabla v^{n-1}_h,\nabla v^{*,n}_h)\\[0.1in]
\qquad  =(f'(u^{n-1}_h)v^{n-1}_h,\chi_2)-(v^{n-1}_h,\chi_2)(f'(u^{n-1}_h)v^{n-1}_h, v^{n-1}_h),\\[0.15in]
\displaystyle v^n_h= v^{*,n}_h/\| v^{*,n}_h\|,
\end{array}
\right.
\end{equation}
with $u_h^0=P u_0$ and $v_h^0=P v_0$. Different from the semi-discrete scheme, a retraction procedure appears in the fully-discrete case as the scheme of $v$ could not preserve the unit length of $v^{*,n}_h$. Concerning this, we define $e_u^{n}:=u(t_n)-u_h^n$, $  e_v^{n}:=v(t_n)-v_h^n$ and $  e_v^{*,n}:=v(t_n)-v_h^{*,n}$
and subtract (\ref{Euler}) from (\ref{app}) to get error equations 
\begin{equation}\label{error}
\left\{
\begin{array}{l}
\displaystyle  \beta^{-1}\big(\frac{e^n_u-e^{n-1}_u}{\tau}+R_u^n, \chi_1\big)  +  (\nabla e^{n}_u, \nabla \chi_1)  +  \mu_1^n  =   \mu_2^n  + \mu_3^n ,\\[0.2in]
\displaystyle \gamma^{-1}\big(\frac{ e^{*,n}_v-e^{n-1}_v}{\tau}+R_v^n, \chi_2\big)+(\nabla e^{*,n}_v, \nabla \chi_2)+\nu_1^n =\nu_2^n +\nu_3^n ,\end{array}
\right.
\end{equation}
where $\displaystyle R_u^n=u_t(t_n)-\frac{u(t_n)-u(t_{n-1})}{\tau}$, $\displaystyle R_v^n=v_t(t_n)-\frac{v(t_n)-v(t_{n-1})}{\tau}$ and
\begin{align*}
\mu_1^n:& =2(\nabla v^{n-1}_h, \nabla u^{n}_h)(v^{n-1}_h, \chi_1)-2\big(\nabla v(t_{n}), \nabla u(t_{n})\big)(v(t_{n}), \chi_1), \\[0.05in]
\mu_2^n:& =2\big(f(u^{n-1}_h), v^{n-1}_h\big)(v^{n-1}_h,\chi_1)-2\big(f(u(t_{n})), v(t_{n})\big)(v(t_{n}),\chi_1),\\[0.05in]
\mu_3^n:& =\big(f(u(t_n))-f(u^{n-1}_h), \chi_1\big),\\[0.05in]
\nu_1^n:&= (v^{n-1}_h,\chi_2)(\nabla v^{*,n}_h,\nabla v^{n-1}_h)- \big(v(t_{n}),\chi_2\big)\big(\nabla v(t_{n}),\nabla v(t_{n})\big), \\[0.05in]
\nu_2^n:&=  \big(f'(u(t_{n}))v(t_{n}),\chi_2 \big)-\big(f'\big(u^{n-1}_h\big)v^{n-1}_h,\chi_2 \big),\\[0.05in]
\nu_3^n:&= (v^{n-1}_h,\chi_2)\big(f'(u^{n-1}_h)v^{n-1}_h, v^{n-1}_h\big)-(v(t_{n}),\chi_2)\big(f'(u(t_{n}))v(t_{n}), v(t_{n})\big).
\end{align*}

%\begin{equation}\label{ev}
%\left\{
%    \begin{aligned}
%\Gamma_1(v_h)&= \big(v^{n-1}_h,\omega\big)\Big(\nabla v^{*,n}_h,\nabla v^{n-1}_h\Big)- \big(v(t_{n-1}),\omega\big)\Big(\nabla v(t_{n}),\nabla v(t_{n-1})\Big), \\[0.05in]
%\Gamma_2(v_h)&=  \Big(f'\big(u(t_{n-1})\big)v(t_{n-1}),\omega \Big)-\Big(f'\big(u^{n-1}_h\big)v^{n-1}_h,\omega \Big),\\[0.05in]
%\Gamma_3(v_h)&= \big(v(t_{n-1}),\omega\big)\Big(f'(u(t_{n-1})v(t_{n-1}), v(t_{n-1})\Big) \\
%& \qquad -\big(v^{n-1}_h,\omega\big)\Big(f'(u^{n-1}_h)v^{n-1}_h, v^{n-1}_h\Big).\\[0.05in]
%\end{aligned}
%    \right.
%\end{equation}

\subsection{Well-posedness issue}
We prove the well-posedness of the solutions to the fully-discrete scheme (\ref{app}) based on an analogous condition in \textit{Assumption B}:

\textit{Assumption C}:  There exist $h_0,\tau_0>0$ such that for any $0<h<h_0$ and $0<\tau<\tau_0$ the solution $u^n_h$ of (\ref{app}), if exists, satisfies 
$\ds	\max_{0\leq n\leq N}\|u_h^n\|_{L^\infty}\!\!\leq \!\!C_1$  for some $C_1>0$.

\begin{theorem}\label{stable}
Under Assumptions A and C, there exists a unique solution to the scheme (\ref{app}) for $\tau$ and $h$ small enough with stability estimates for $1\leq n\leq N$
\begin{align}\label{BNablaV}
		\|\nabla u_h^n\|+\|\nabla v_h^n\|+\|\nabla v_h^{*,n}\|\leq Q,~~\big|\|v_h^{*,n}\|-1\big|\leq \big|\|v_h^{*,n}\|^2-1\big|\leq Q\tau.
\end{align}
\end{theorem}

\begin{proof}
We divide the induction procedure into two steps. In the first step, we assume that the numerical solutions exist for $1\leq n\leq m$ for some $m>0$ (the initial values, which corresponds to $n=0$, always exist) and then derive the stability estimates for $1\leq n\leq m$. In the second step, we invoke these stability estimates to show the existence and uniqueness of the numerical solutions at $n=m+1$. It is worth mentioning that the bound $Q$ in the stability estimate (\ref{BNablaV}) with $1\leq n\leq m$ is independent from $n$ and $m$ such that we do not need to worry about the increment of the bound during the induction procedure.

\textit{Step 1: Stability estimates for $1\leq n
\leq m$.}  Take $\chi_2 =v_h^{n-1}$ in the second equation of (\ref{app}) to obtain 
\begin{equation}\label{norm1}
	\begin{aligned}
		(v_h^{*,n},v_h^{n-1})=1~ (\text{such that }(v_h^{*,n}-v_h^{n-1},v_h^{n-1})=0) .
	\end{aligned}
\end{equation}
We apply this to get
\begin{equation}\label{norm3}
	\|v_h^{*,n}-v_h^{n-1}\|^2=\|v_h^{*,n}\|^2-2(v_h^{*,n},v_h^{n-1})+1=\|v_h^{*,n}\|^2-1.
\end{equation}
As $ v_h^{*,n}=v_h^n \|v_h^{*,n}\|$, we have $\nabla v_h^{*,n}=\nabla v_h^n \|v_h^{*,n}\|$, which, together with (\ref{norm3}), leads to 
\begin{align}\label{yx1}
\nabla v_h^{*,n}=\nabla v_h^n (1+\|v_h^{*,n}-v_h^{n-1}\|^2)^{1/2}.
\end{align}
Then we select $\chi_2=v_h^{*,n}-v_h^{n-1}$ and incorporate (\ref{norm1}) to get
\begin{equation*}
	\gamma^{-1}\|v_h^{*,n}-v_h^{n-1}\|^2+\tau\big(\nabla v_h^{*,n},\nabla (v_h^{*,n}-v_h^{n-1})\big)=\tau (f'(u^{n-1}_h)v^{n-1}_h, v_h^{*,n}-v_h^{n-1}).
\end{equation*}
An application of the Cauchy inequality  yields
\begin{equation*}
\gamma^{-1}\|v_h^{*,n}-v_h^{n-1}\|^2+\frac{\tau}{2}\|\nabla v_h^{*,n}\|^2 \leq \frac{\tau}{2}\|\nabla v_h^{n-1}\|^2 +\tau (f'(u^{n-1}_h)v^{n-1}_h, v_h^{*,n}-v_h^{n-1}).
\end{equation*}
 We invoke (\ref{yx1}) in the above equation to get
\begin{equation*}
	\gamma^{-1}\|v_h^{*,n}-v_h^{n-1}\|^2 + \frac{\tau}{2} \|\nabla v_h^n\|^2 \big(1+\|v_h^{*,n}-v_h^{n-1}\|^2 \big) \leq \frac{\tau}{2} \|\nabla v_h^{n-1}\|^2 +Q\tau \|v_h^{*,n}-v_h^{n-1}\|.
\end{equation*}
We apply the Young’s inequality $\displaystyle Q\tau \|v_h^{*,n}-v_h^{n-1}\|\leq Q\tau^2+\frac{\gamma^{-1}}{2}\|v_h^{*,n}-v_h^{n-1}\|^2$ to get
\begin{equation}\label{bb1}
	\gamma^{-1}\|v_h^{*,n}-v_h^{n-1}\|^2+\tau\|\nabla v_h^n\|^2 \leq \tau\|\nabla v_h^{n-1}\|^2+Q\tau^2.
\end{equation}
Sum over this equation from $n = 1$ to $n^* \leq m$ to obtain
$
	\|\nabla v_h^{n^*}\|^2 \leq \|\nabla v_h^{0}\|^2+ Q\tau n^* \leq Q\|\nabla v_{0}\|^2+ QT\leq Q,
$
 and we invoke this estimate back to (\ref{bb1}) to get $\|v_h^{*,n}-v_h^{n-1}\|^2\leq Q\tau$, which, together with (\ref{norm3}), gives $\big|\|v_h^{*,n}\|^2-1\big|\leq Q\tau$. We invoke these estimates and $\nabla v_h^{*,n}=\nabla v_h^{n}\|v_h^{*,n}\| $ to  immediately get $\|\nabla v_h^{*,n}\|\leq Q$.

To bound $\nabla u_h^{n}$,
we set $\chi_1 =u_h^n-u_h^{n-1}$ in (\ref{app}) and invoke the boundedness of $\|u_h^n\|$ and $\|\nabla v_h^n\|$ to derive
\begin{equation*}
	\begin{aligned}
&\beta^{-1}\|u_h^n-u_h^{n-1}\|^2 + \tau\|\nabla u_h^n\|^2 \\
&\qquad \leq \tau\|\nabla u_h^n\|\cdot \|\nabla u_h^{n-1}\|+2\tau\|\nabla v_h^{n-1}\|\cdot\|\nabla u_h^{n}\|\cdot\|u_h^n-u_h^{n-1}\| + Q\tau\|u_h^n-u_h^{n-1}\|\\
&\qquad \leq  Q\tau^2\|\nabla u_h^{n}\|^2+\beta ^{-1}\|u_h^n-u_h^{n-1}\|^2+\frac{\tau}{2}\|\nabla u_h^n\|^2  +\frac{\tau}{2}\|\nabla u_h^{n-1}\|^2+Q\tau^2,
\end{aligned}
\end{equation*}
which implies
$
	\|\nabla u_h^n\|^2-\|\nabla u_h^{n-1}\|^2\leq Q\tau\|\nabla u_h^{n}\|^2+Q\tau.
$
We sum this equation with respect to $n$ and apply the Gronwall inequality to get $\|\nabla u_h^n \| \leq Q $. %Thus (\ref{BNablaV}) is proved.
  
\textit{Step 2: Well-posedness of scheme at $n=m+1$.}
We rewrite the scheme (\ref{app})  as
\begin{equation*}
\mathcal{A}_{m+1}(u_h^{m+1},\chi_1)=\big(\mathcal{F}^{m}, \chi_1\big),~~ \mathcal{B}_{m+1}(v_h^{*,m+1}, \chi_2)=\big(\mathcal{G}^{m}, \chi_2\big ),~~ v_h^{m+1} =\frac{ v_h^{*,m+1}}{\| v_h^{*,m+1}\|},
\end{equation*}
where bilinear forms $\mathcal{A}_{m+1}(\cdot,\cdot), \mathcal{B}_{m+1}(\cdot,\cdot):\,(H^1_0)^2\rightarrow \mathbb R$ and $\mathcal{F}^{m}$, $\mathcal{G}^{m}$ are given by
\begin{equation*}
	\begin{aligned}
\mathcal{A}_{m+1}(p, q)& :=\beta^{-1}(p, q)+\tau(\nabla p, \nabla q)-2\tau(\nabla p, \nabla v_h^{m})(v_h^{m}, q),~~p,q\in H^1_0,\\[0.05in]
\mathcal{F}^{m}& := u_h^{m}+\tau f(u_h^{m})-2\tau (f(u_h^{m}),v_h^{m})v_h^{m},\\[0.05in]
\mathcal{B}_{m+1}(p, q)& :=\gamma^{-1}(p,q)+\tau (\nabla p, \nabla q)-\tau (\nabla p, \nabla v_h^{m})(v_h^{m}, q),~~p,q\in H^1_0,\\[0.05in]
\mathcal{G}^{m}& :=v_h^{m}+\tau f'(u^{m}_h)v^{m}_h-\tau(f'(u^{m}_h)v^{m}_h, v^{m}_h)v^{m}_h.
	\end{aligned}
\end{equation*}
Based on the stability (\ref{BNablaV}) with $1\leq n\leq m$ and  \textit{Assumptions A} and \textit{C}, it is clear that $\mathcal{F}^{m}, \mathcal{G}^{m}\in L^2$ and
		$|\mathcal{A}_{m+1}(p,q)|\leq Q\|p\|_{H^1}\|q\|_{H^1}$, and 
\begin{equation*}
	\begin{aligned}
\mathcal{A}_{m+1}(p, p) &\geq \beta^{-1}\|p\|^2+ \tau \|\nabla p\|^2-Q\tau \|p\|\cdot \|\nabla p\|\\
	& \geq \beta^{-1}\|p\|^2+ \tau \|\nabla p\|^2-Q\tau \|p\|^2-\frac{\tau}{2}\|\nabla p\|^2.
	\end{aligned}
\end{equation*} 
Thus, if $\tau$ small enough such that $\beta^{-1}-Q\tau >0$, we get the coercivity.
We could similarly analyze $\mathcal{B}_n(\cdot, \cdot)$.
Consequently, we apply the Lax-Milgram Theorem to obtain the existence and uniqueness of the solutions to (\ref{app}) at $n=m+1$. Then by induction we reach the conclusions of the theorem.
\end{proof}

\subsection{Analysis of normalized projection}\label{sec43}
In conventional numerical analysis of finite element method for evolution equations, one usually introduce the interpolation or elliptic projection of solutions as an intermediate variable to split the errors and then perform estimates.
In error equation (\ref{error}), however, both $e_v^{n}$ and $e_v^{*,n}$ are presented due to the normalization and thus the  intermediate variable needs to be carefully designed  to facilitate numerical analysis.
For this purpose, we introduce the normalized projection  $\widetilde{Pv(t_n)}:= Pv(t_n)/\|Pv(t_n)\|$ and split the errors as
 \begin{align*}
 &e_v^n=\displaystyle v(t_n)-v_h^n=\big(v(t_n)-\widetilde{Pv(t_n)}\big)+\big(\widetilde{Pv(t_n)}-v_h^n\big):=\eta_v^n+\xi_v^n,\\
 &e_v^{*,n}=v(t_n)-v_h^{*,n}=\big(v(t_n)-\widetilde{Pv(t_n)}\big)+\big(\widetilde{Pv(t_n)}-v_h^{*,n}\big):=\eta_v^n+\xi_v^{*,n}.
 \end{align*}
 For $n=0$, we set $v_h^{*,0}=\widetilde{Pv_0}$ such that $\xi_v^{*,0}=0$.
 
 Now we will show that such splitting admits the following implicit relation that plays a key role in error estimates.

\begin{lemma}\label{li}
	If $\big(\widetilde{Pv(t_n)}, \widetilde{Pv(t_n)}-v_h^n\big)\leq 1$, then 
	$
		\|\xi_v^n\|\leq \|\xi_v^{*,n}\|+\|\xi_v^{*,n}\|^3.
	$
\end{lemma}
\begin{proof}
Let $\displaystyle \widehat{Pv(t_n)}=\big(\widetilde{Pv(t_n)}, v_h^n\big)v_h^n$ such that $\|\widetilde{Pv(t_n)}-\widehat{Pv(t_n)}\|\leq \|\widetilde{Pv(t_n)}-lv_h^n\|$ for any $l\in\mathbb R$. We select $l=\|v_h^{*,n}\|$ and recall $v_h^{*,n}=\|v_h^{*,n}\|v_h^n$ to get
\begin{align}\label{eq3}
		\|\widetilde{Pv(t_n)}-v_h^n\|-\|\widetilde{Pv(t_n)}-v_h^{*,n}\|&\leq \|\widetilde{Pv(t_n)}-v_h^n\|-\|\widetilde{Pv(t_n)}-\widehat{Pv(t_n)}\|\notag\\
		&=\|\widetilde{Pv(t_n)}-v_h^n\|\Big( 1-\frac{\|\widetilde {Pv(t_n)}-\widehat{Pv(t_n)}\|}{\|\widetilde{Pv(t_n)}-v_h^n\|}\Big).
	\end{align}
	As we suppose $\big(\widetilde{Pv(t_n)}, \widetilde{Pv(t_n)}-v_h^n\big)\leq 1$ in this lemma, we have 
\begin{align}\label{bb3}
0\leq 1-\Big(\widetilde{Pv(t_n)}, \widetilde{Pv(t_n)}-v_h^n\Big)=\big(\widetilde{Pv(t_n)}, v_h^n\big)\leq 1.
\end{align}
	We invoke this, $( \widehat{Pv(t_n)},v_h^n)=\big(\widetilde{Pv(t_n)}, v_h^n\big)$ and $\big(\widetilde{Pv(t_n)}, \widehat{Pv(t_n)}\big)=\big(\widetilde{Pv(t_n)}, v_h^n\big)^2=\|\widehat{Pv(t_n)}\|^2$ to get
\begin{align*}
	\|\widetilde{Pv(t_n)}-\widehat{Pv(t_n)}\|^2 &=1-2\big(\widetilde{Pv(t_n)}, \widehat{Pv(t_n)}\big)+\|\widehat{Pv(t_n)}\|^2=1-\big(\widetilde{Pv(t_n)}, v_h^n\big)^2\\
	&\geq \frac{1-\big(\widetilde{Pv(t_n)}, v_h^n\big)}{2} =\frac{\big\|\widetilde{Pv(t_n)}-v_h^n\big\|^2}{4},
\end{align*}
which, together with $\|\widetilde{Pv(t_n)}-\widehat{Pv(t_n)}\|\leq \|\widetilde{Pv(t_n)}-v_h^{*,n}\|$, leads to 
\begin{equation}\label{first}
	\frac{1}{2}\|\widetilde{Pv(t_n)}-v_h^n\| \leq \|\widetilde{Pv(t_n)}-\widehat{Pv(t_n)}\|\leq \|\widetilde{Pv(t_n)}-v_h^{*,n}\|.
\end{equation}
For $y\in [0,1]$, we have $y^2+1\leq y+1\leq \sqrt{2(y+1)}$, that is, $1-\sqrt{(1+y)/2}\leq (1-y^2)/2$. We apply this, $\big(\widetilde{Pv(t_n)}, \widehat{Pv(t_n)}\big)=\big(\widetilde{Pv(t_n)}, v_h^n\big)^2=\|\widehat{Pv(t_n)}\|^2$ and   (\ref{bb3})  to get
\begin{align}\label{second}
	1-\frac{\|\widetilde{Pv(t_n)}-\widehat{Pv(t_n)}\|}{\|\widetilde{Pv(t_n)}-v_h^n\|}& =1-\sqrt{\frac{1-\big(\widetilde{Pv(t_n)},v_h^n\big)^2}{2\big(1-(\widetilde{Pv(t_n)},v_h^n)\big)}}=1-\sqrt{\frac{1+(\widetilde{Pv(t_n)},v_h^n)}{2}}\nonumber\notag\\
	&\leq \frac{1-(\widetilde{Pv(t_n)},v_h^n)^2}{2}=\frac{1}{2}\|\widetilde{Pv(t_n)}-\widehat{Pv(t_n)}\|^2.
\end{align}
%Furthermore,  we observe that the term behind the second equality of (\ref{second}), and thus the left-hand side of (\ref{second}), is non-negative.
Invoking (\ref{first})--(\ref{second}) in (\ref{eq3}) completes our proof.
\end{proof}

\subsection{Auxiliary estimates}We derive several estimates based on \textit{Assumptions A} and \textit{C} (such that Theorem \ref{stable} could be applied) and the regularity condition $u,v\in H^1(0,T;H^2)\cap H^2(0,T;L^2)$ (such that $u$ and $v$ belong to $ C([0,T];L^\infty)$ and $ W^{1,\infty}(0,T;L^2)$).

\underline{Estimates related to $u$} We set $\chi_1=\xi_u^n$ in the first equation of (\ref{error}) to get
\begin{equation*}%\label{u1}
	\begin{aligned}
|\mu_1^n|& = \big|2\big(\nabla v(t_{n}), \nabla u(t_{n})\big)(v(t_{n}), \xi_u^n)-2(\nabla v^{n-1}_h, \nabla u^{n}_h)(v^{n-1}_h, \xi_u^n)\big|\\
&=2\big|\big(\nabla v(t_{n})-\nabla v(t_{n-1})+\nabla e_v^{n-1}, \nabla u(t_{n})\big)(v(t_{n}), \xi_u^n)\\
&\quad+ (\nabla v_h^{n-1}, \nabla e_u^{n})\big(v(t_{n}), \xi_u^n\big)+ (\nabla v_h^{n-1}, \nabla u_h^n)\big(v(t_{n})-v(t_{n-1})+e_v^{n-1}, \xi_u^n\big)\big|\\
&\leq Q\|e_v^{n-1}\|\cdot \|\xi_u^n\|+Q\|\nabla \xi_u^{n}\|\cdot \|\xi_u^n\|+Q\tau \|\xi_u^n\|,
	\end{aligned}
\end{equation*}
where we have used the integration by parts 
$\big(\nabla v(t_{n})-\nabla v(t_{n-1})+\nabla e_v^{n-1}, \nabla u(t_{n})\big)=-\big( v(t_{n})- v(t_{n-1})+ e_v^{n-1}, \Delta u(t_{n})\big)$.
Similarly, we bound $\mu_2^n$ and $\mu_3^n$ as
\begin{equation*}%\label{u2}
	\begin{aligned}
|\mu_2^n|& =\big|2\big(f(u(t_{n})), v(t_{n})\big)(v(t_{n}),\xi_u^n)-2\big(f(u^{n-1}_h), v^{n-1}_h\big)(v^{n-1}_h,\xi_u^n)\big|\\
&=2\big|\big(f(u(t_{n}))-f(u(t_{n-1}))+f(u(t_{n-1}))-f(u^{n-1}_h), v(t_{n})\big)(v(t_{n}),\xi_u^n)\\
&\quad+\big(f(u_h^{n-1}), v(t_{n})-v(t_{n-1})+e_v^{n-1}\big)(v(t_{n}),\xi_u^n)\\
&\quad +\big(f(u_h^{n-1}), v_h^{n-1}\big)(v(t_{n})-v(t_{n-1})+e_v^{n-1},\xi_u^n)\big|\\
&\leq Q\|u(t_{n-1})-u_h^{n-1}\|\cdot \|\xi_u^n\|+Q\|e_v^{n-1}\|\cdot \|\xi_u^n\|+Q\tau \|\xi_u^n\|,
	\end{aligned}
\end{equation*}
and $$
|\mu_3^n| =\big|\big(f(u(t_n))-f(u^{n-1}_h), \xi_u^n \big)\big|
		\leq Q\|e_u^{n-1}\|\cdot \|\xi_u^n\|+Q\tau \|\xi_u^n\|.$$
Invoke the above three estimates into the first equation of (\ref{error}) to derive
\begin{align}
\beta^{-1}\|\xi_u^n\|^2+\tau\|\nabla\xi_u^n\|^2 & \leq \beta^{-1} \|\xi_u^{n-1}\|\cdot\|\xi_u^{n}\|+Q\tau \|\nabla \xi_u^{n}\| \cdot\|\xi_u^{n}\|+\tau\|\psi^n_u\|\cdot\|\xi_u^{n}\|\notag\\[0.05in]
 &\quad+ Q\tau^2 \|\xi_u^n\|+Q\tau\|e_u^{n-1}\|\cdot\|\xi_u^{n}\| + Q\tau\| e_v^{n-1}\|\cdot\|\xi_u^{n}\|,\label{errorFullU}
\end{align}
where $\psi^n_u=-\beta^{-1}(\frac{\eta_u^n-\eta_u^{n-1}}{\tau}+R^n_u)$. By the Young’s inequality $Q \|\nabla \xi_u^{n}\| \cdot\|\xi_u^{n}\|\leq \|\nabla \xi_u^{n}\|^2+Q\|\xi_u^{n}\|^2$ in (\ref{errorFullU}), we cancel the  $\|\nabla \xi_u^{n}\|^2$ on both sides of (\ref{errorFullU}) and then cancel $\|\xi_u^n\|$ in each term to get
$$
	\|\xi_u^n\| \leq  \|\xi_u^{n-1}\|+Q\tau\|\xi_u^{n}\|+\tau\|\psi^n_u\|+ Q\tau^2 +Q\tau\|e_u^{n-1}\| + Q\tau\| e_v^{n-1}\|.
$$
We sum this equation from $n=1$ to $n^*\leq N$ and invoke the standard estimate
(see e.g. \cite[Equations 6.18 and 9.12]{ZheSICON}) 
\begin{equation*}
	\tau \sum_{n=1}^{N}\|\psi^n_u\|\leq Q\|u_{tt}\|_{L^1(0,T;L^2)}\tau+Q\|u_t\|_{L^1(0,T;H^2)}h^2,
\end{equation*}
 to get
$$\ds
	\|\xi_u^{n^*}\|\leq Q\tau \sum_{n=1}^{n^*}\|\xi_u^{n}\|+ Q\tau\sum_{n=1}^{n^*}\|e_v^{n-1}\|+Q\tau+Qh^2.
$$
Then for $\tau$ small enough such that $Q\tau<1$, we apply the Gronwall inequality to get
\begin{align}\label{bb5}
	\|\xi_u^{n}\|\leq  Q\tau\sum_{m=1}^{n}\|e_v^{m-1}\|+Q\tau+Qh^2,~~1\leq n\leq N.
\end{align}

\underline{Estimates related to $v$} We set $\chi_2=\xi_v^{*,n}$ in the second equation of (\ref{error}) and similarly bound $\nu_1^n$--$\nu_3^n$ as $\mu_1^n$--$\mu_3^n$
\begin{align*}
	|\nu_1^n|&=|(v(t_n),\xi_v^{*,n})(\nabla v(t_n),\nabla v(t_n))-(v_h^{n-1},\xi_v^{*,n})(\nabla v_h^{*,n}, \nabla v_h^{n-1})|\\
		&=|(v(t_n)-v(t_{n-1})+e_v^{n-1},\xi_v^{*,n})(\nabla v(t_n),\nabla v(t_n))\\
		&\quad +(v_h^{n-1},\xi_v^{*,n})(\nabla (v(t_n)-v(t_{n-1}))+ \nabla e_v^{n-1},\nabla v(t_n))\\
		&\quad +(v_h^{n-1},\xi_v^{*,n})(\nabla v_h^{n-1},\nabla e_v^{*,n})|\\
		&\leq Q\|e_v^{n-1}\|\cdot\|\xi_v^{*,n}\|+Q\|\nabla \xi_v^{*,n}\|\cdot\|\xi_v^{*,n}\|+Q\tau \|\xi_v^{*,n}\|;\\
|\nu_2^n|&=\big|\big(f'(u(t_{n}))v(t_{n}), \xi_v^{*,n} \big)-\big(f'(u^{n-1}_h)v^{n-1}_h,\xi_v^{*,n}\big)\big|\\
&=\big|\big((f'(u(t_{n}))-f'(u(t_{n-1}))+f'(u(t_{n-1}))-f'(u^{n-1}_h))v(t_{n}),\xi_v^{*,n} \big)\\
&\quad +\big(f'(u^{n-1}_h)(v(t_{n})-v(t_{n-1})+e_v^{n-1}), \xi_v^{*,n} \big)\big|\\
&\leq Q\|u(t_{n-1})-u_h^{n-1}\|\cdot \|\xi_v^{*,n}\|+Q\|e_v^{n-1}\|\cdot\|\xi_v^{*,n}\|+Q\tau \|\xi_v^{*,n}\|;\\
|\nu_3^n|& = \big|(v(t_{n}),\xi_v^{*,n})\big(f'(u(t_{n}))v(t_{n}), v(t_{n})\big)-(v^{n-1}_h,\xi_v^{*,n})\big(f'(u^{n-1}_h)v^{n-1}_h, v^{n-1}_h\big)\big|\\
& \leq Q\|u(t_{n-1})-u_h^{n-1}\|\cdot \|\xi_v^{*,n}\|+Q\|e_v^{n-1}\|\cdot\|\xi_v^{*,n}\|+Q\tau \|\xi_v^{*,n}\|.
\end{align*}
We invoke these in the second equation of (\ref{error}) and apply
$$ \big(\nabla e_v^{*,n},\nabla \xi_v^{*,n}\big)=\big(\delta_v^n,\nabla \xi_v^{*,n}\big)+\|\nabla\xi_v^{*,n}\|^2\geq \frac{1}{2} \|\nabla\xi_v^{*,n}\|^2-Q\|\delta_v^n\|^2$$ 
where $\delta_v^n:=\nabla (Pv(t_n)-\widetilde{Pv(t_n)})$ to get
\begin{equation}\label{zjy1}
	\begin{aligned}
&\gamma^{-1}\|\xi_v^{*,n}\|^2+\frac{\tau}{2}\|\nabla \xi_v^{*,n} \|^2  \leq  \gamma^{-1}\|\xi_v^{n-1}\|\cdot\|\xi_v^{*,n}\|+Q\tau \|\delta _v^n\|^2+Q\tau\|\psi^n_v\|\cdot\|\xi_v^{*,n}\|\\
&~+Q\tau \|\nabla \xi_v^{*,n}\|\cdot\|\xi_v^{*,n}\| +Q\tau \|e_u^{n-1}\|\cdot \|\xi_v^{*,n}\|  +Q\tau \|e_v^{n-1}\|\cdot\|\xi_v^{*,n}\|+Q\tau^2 \|\xi_v^{*,n}\|\\
&\leq \frac{\gamma^{-1}}{2}\|\xi_v^{n-1}\|^2+\frac{\gamma^{-1}}{2}\|\xi_v^{*,n}\|^2+Q\tau \|\delta _v^n\|^2+Q\tau\|\psi^n_v\|^2\\
&~+\frac{\tau}{2}\|\nabla \xi_v^{*,n}\|^2+Q\tau\|e_u^{n-1}\|^2+Q\tau\|e_v^{n-1}\|^2 +Q\tau^2\|\xi_v^{*,n}\|+Q\tau\|\xi_v^{*,n}\|^2,
	\end{aligned}
\end{equation}
where $\psi^n_v=-\gamma^{-1}(\frac{\eta_v^n-\eta_v^{n-1}}{\tau}+R^n_v)$.
\begin{lemma}\label{estimateB}
If $v\in H^1(0,T;H^2)\cap H^2(0,T;L^2)$, then
  $
		\tau \sum_{n=1}^{N}\|\psi^n_v\|^2   \leq  Q(\tau^2+h^4)
	$ for $\tau$ small enough and $\|\eta_v^n\|+\|\delta _v^n\|\leq Qh^2$ for $1\leq n\leq N$.
\end{lemma}
\begin{proof}
Through $\big|\|Pv(t_n)\|-1\big|=\big|\|Pv(t_n)\|-\|v(t_n)\|\big|\leq \|Pv(t_n)-v(t_n)\|\leq Qh^2\|v(t_n)\|_{H^2}$, we have $\|Pv(t_n)\|\geq \frac{1}{2}$ for $h$ small enough. Then we have
 \begin{equation*}\label{bb9}
 	 \begin{aligned}
	\|\eta_v^n\|&=\frac{1}{\|Pv(t_n)\|}\big\|v(t_n)\|Pv(t_n)\|-Pv(t_n)\big\| \\[0.05in]
	& = \frac{\big\| v(t_n)\|Pv(t_n)\|- Pv(t_n)\|Pv(t_n)\| +Pv(t_n)\|Pv(t_n)\|- Pv(t_n)\big\|}{\|Pv(t_n)\|} \\[0.05in]
	&\leq \|v(t_n)-Pv(t_n)\|+\big|\|Pv(t_n)\|-1\big|\ \leq 2\|v(t_n)-Pv(t_n)\|\leq Qh^2.
\end{aligned}
 \end{equation*} 
By $\|Pv(t_n)\|\geq \frac{1}{2}$ and $\|\nabla Pv(t_n)\|\leq \|\nabla v(t_n)\|$, we have
\begin{equation*}%\label{delta}
	\|\delta_v^n\|  =\frac{\|\nabla Pv(t_n)\|}{\|Pv(t_n)\|}\cdot \big|\|Pv(t_n)\|-1\big| \leq Q\|Pv(t_n)-v(t_n)\| \leq Q h^2.
 \end{equation*}
The $R_v^n$ in $\psi_v^n$ could be bounded as
\begin{equation*}
	\tau \sum_{n=1}^{N}\|R^n_v\|^2=\tau \sum_{n=1}^{N}\Big\|\frac{1}{\tau}\int_{t_{n-1}}^{t_n}\int_s^{t_n}v_{yy}(y)dyds\Big\|^2\leq Q\|v_{t}\|^2_{L^2(0,T;L^2)}\tau^2.
\end{equation*}
To estimate $\eta_v^n-\eta_v^{n-1}$ in $\psi_v^n$, we employ the splitting
 \begin{equation*}
 	\begin{aligned}
\eta_v^n-\eta_v^{n-1}&=\Big(v(t_n)-\frac{Pv(t_n)}{\|Pv(t_n)\|}\Big)-\Big(v(t_{n-1})-\frac{Pv(t_{n-1})}{\|Pv(t_{n-1})\|}\Big)\\
		&\hspace{-0.4in}=\Big[\big(v(t_n)-Pv(t_n)\big)-\big(v(t_{n-1})-Pv(t_{n-1})\big)\Big]\\
		&\hspace{-0.4in}\quad+\Big[\frac{Pv(t_n)}{\|Pv(t_n)\|}\big(\|Pv(t_n)\|-1\big)- \frac{Pv(t_{n-1})}{\|Pv(t_{n-1})\|}\big(\|Pv(t_{n-1})\|-1\big) \Big]=:A_1+A_2.
 	\end{aligned}
 \end{equation*}
We bound $A_1$ following the standard manner
\begin{equation*}
	\tau ^{-1}\sum_{n=1}^N \|A_1\|^2=\tau ^{-1}\sum_{n=1}^N \Big\|\int _{t_{n-1}}^{t_{n}} (I-P)v_s(s)ds \Big\|^2 \leq Q\|v_{t}\|^2_{L^2(0,T;H^2)}h^4.
\end{equation*}
To bound $A_2$, we employ a further splitting
\begin{equation*}
\begin{aligned}
|A_2| 
	&\leq \Big|\frac{Pv(t_n)}{\|Pv(t_n)\|}\big(\|Pv(t_n)\|-1-(\|Pv(t_{n-1})\|-1))\Big|\\
	&\quad +\Big|\big(\|Pv(t_{n-1})\|-1\big)\Big(\frac{Pv(t_{n})}{\|Pv(t_{n}))\|}-\frac{Pv(t_{n-1})}{\|Pv(t_{n-1}))\|}\Big) \Big|=:B_1+B_2.
	\end{aligned}
\end{equation*}
We apply $ \|Pv_s\|\leq \|v_s\|+Q\|v_s\|_{H^2}h^2\leq Q\|v_s\|_{H^2} $ to get
\begin{equation*}
	\begin{aligned}
\|B_1\|& \leq \big|\|Pv(t_n)\|-\|v(t_n)\|-(\|Pv(t_{n-1})\|-\|v(t_{n-1})\|)\big|\\[0.05in]
	&=\Big| \int_{t_{n-1}}^{t_n}\!\! \big(\|Pv(s)\|-\|v(s)\|\big)_s ds \Big|\leq \int_{t_{n-1}}^{t_n} \Big| \frac{\big(Pv(s), Pv_s(s)\big)}{\|Pv(s)\|}\!-\!\frac{\big(v(s), v_s(s)\big)}{\|v(s)\|}\Big|ds\\[0.05in]
	&=  \int_{t_{n-1}}^{t_n} \Big|\frac{\big(Pv(s)-v(s), Pv_s(s)\big)}{\|Pv(s)\|}+\frac{\big( Pv_s(s)-v_s(s),v(s)\big)}{\|Pv(s)\|}\\[0.05in]
 &+\big(v(s), v_s(s)\big)\frac{\|v(s)\|\!-\!\|Pv(s)\|}{\|v(s)\|\!\cdot\!\|Pv(s)\|}\Big|ds\leq Qh^2\!\! \int_{t_{n-1}}^{t_n}\!\! \!\! \|Pv_s(s)\|+\|v\|_{H^2}+\|v_s\|_{H^2}ds\\
& \leq Qh^2 \int_{t_{n-1}}^{t_n}\|v\|_{H^2}+\|v_s\|_{H^2}ds,
	\end{aligned}
\end{equation*}
leading to $\tau ^{-1}\sum_{n=1}^N \|B_1\|^2\leq Qh^4\|v\|^2_{H^1(0,T;H^2)}$. We reformulate $B_2$ as
\begin{equation*}
	B_2= \big|\|Pv(t_{n-1})\|-\|v(t_{n-1})\| \big| \cdot \Big|\int_{t_{n-1}}^{t_n}\Big(\frac{Pv(s)}{\|Pv(s)\|}\Big)_s ds\Big|.
\end{equation*}
We incorporate $\|Pv\|\leq \|v\|+Qh^2\|v\|_{H^2}\leq Q$ to get
 \begin{equation*}
 	\begin{aligned}
 		\Big\|\Big(\frac{Pv(s)}{\|Pv(s)\|}\Big)_s\Big\|=\Big\|\frac{P  v_s(s)}{\|Pv(s)\|} - \frac{Pv(s) (Pv,Pv_s)}{\|Pv(s)\|^3}\Big\|\leq Q\|v_s\|_{H^2},
 	\end{aligned}
 \end{equation*}
 which implies  $\tau ^{-1}\sum_{n=1}^N \|B_2\|^2\leq Qh^4\|v\|^2_{H^1(0,T;H^2)}$.
We combine the above estimates to complete the proof.
\end{proof}

We invoke $\|e_v^{n-1}\|\leq \|\eta_v^{n-1}\|+\|\xi_v^{n-1}\|$, (\ref{bb5}), $Q\tau^2\|\xi_v^{*,n}\|\leq Q\tau\|\xi_v^{*,n}\|^2+Q\tau^3$ and Lemma \ref{estimateB} in (\ref{zjy1}) to get
\begin{align*}
&\|\xi_v^{*,n}\|^2
\leq \|\xi_v^{n-1}\|^2+Q\tau\|\xi_v^{n-1}\|^2\nonumber\\
& ~\quad+Q\tau^2\sum_{m=1}^n\|\xi_v^{m-1}\|^2 +Q\tau\|\xi_v^{*,n}\|^2+Q\tau\|\psi^n_v\|^2+Q\tau(\tau^2+h^4).
\end{align*}
Sum this equation from $n=1$ to $n^*\leq N$ to get
\begin{align*}
&\sum_{n=1}^{n^*}\|\xi_v^{*,n}\|^2
\leq \sum_{n=1}^{n^*}\|\xi_v^{n-1}\|^2+Q\tau\sum_{n=1}^{n^*}\|\xi_v^{n-1}\|^2+Q\tau\sum_{n=1}^{n^*}\|\xi_v^{*,n}\|^2+Q(\tau^2+h^4),
\end{align*}
that is, for $\tau$ small enough,
\begin{align}
\sum_{n=1}^{m}\|\xi_v^{*,n}\|^2
&\leq \sum_{n=1}^{m}\|\xi_v^{n-1}\|^2+\frac{2Q\tau}{1-Q\tau}\sum_{n=1}^{m}\|\xi_v^{n-1}\|^2+Q(\tau^2+h^4)\nonumber\\
&\leq \sum_{n=1}^{m}\|\xi_v^{n-1}\|^2+\tilde Q\tau\sum_{n=1}^{m}\|\xi_v^{n-1}\|^2+\tilde Q(\tau^2+h^4),~~1\leq m\leq N,\label{zjy4}
\end{align}
for some constant $\tilde Q>0$.
\subsection{Error estimates}
We derive error estimates based on an induction procedure in the following theorem.
\begin{theorem}\label{thm44}
Under Assumptions A and C and the regularity condition $u,v\in H^1(0,T;H^2)\cap H^2(0,T;L^2)$, the following error estimate  holds for $1\leq n\leq N$
$$\|u(t_n)-u_h^n\|+\|v(t_n)-v_h^n\|\leq Q(\tau+h^2)$$
 if $h$ and $\tau$ are small enough and satisfy $h=O(\tau^{\frac{1+\varepsilon}{4}})$ for $0<\varepsilon\ll 1$.
\end{theorem} 
\begin{remark}\label{rem4}
Similar to Corollary \ref{cor1}, the error estimate indicates that $u_h^N$ could approximate the index-1 saddle point $u^*$ with an arbitrarily small $L^2$ error such that the fully-discrete scheme (\ref{app}) is index-preserving for $T$ large enough and $\tau$ and $h$ small enough.
\end{remark}
\begin{proof}
%\begin{equation*}
%	\begin{aligned}
%\|\xi_v^{*,n}\|^2 & \leq \|\xi_v^{*,n-1}\|^2+2\|\xi_v^{*,n-1}\|^4+\|\xi_v^{*,n-1}\|^6+Q\tau \|\delta _v^n\|^2+Q\tau\|\psi^n_v\|^2\\
%&  +Q\tau(\|\eta_v^{n-1}\|^2+\|\xi_v^{*,n-1}\|^2+\|\xi_v^{*,n-1}\|^6)\\
%&+Q\tau^2\sum_{m=1}^{n}(\|\eta_v^{m-1}\|^2+\|\xi_v^{*,m-1}\|^2+\|\xi_v^{*,m-1}\|^6)+Q\tau\|\xi_v^{*,n}\|^2+Q\tau(\tau^2+ h^4).
%	\end{aligned}
%\end{equation*}
%Sum this equation from $n=1$ to $n^*\leq N$ and apply Gronwall's inequality to get
%\begin{equation}\label{bb6}
%	\begin{aligned}
%\|\xi_v^{*,n^*}\|^2 & \leq Q\sum_{n=1}^{n^*}\big(\|\xi_v^{*,n-1}\|^4+\|\xi_v^{*,n-1}\|^6\big)\\
%&\quad+Q\tau \sum_{n=1}^{n^*}\big(\|\delta _v^n\|^2+\|\psi^n_v\|^2  +\|\eta_v^{n-1}\|^2\big)+Q(\tau^2+ h^4).
%	\end{aligned}
%\end{equation}
%
%
%
%We invoke Lemma \ref{estimateB} in (\ref{bb6}) to get
%\begin{equation}\label{bb7}
%\|\xi_v^{*,n^*}\|^2  \leq \tilde Q\sum_{n=1}^{n^*}\big(\|\xi_v^{*,n-1}\|^4+\|\xi_v^{*,n-1}\|^6\big)+\tilde Q(\tau^2+ h^4).
%\end{equation}
In Section \ref{sec43} we have set $v_h^{*,0}=\widetilde{Pv_0}$ such that $\xi_v^{*,0}=0$. As $h=O(\tau^{\frac{1+\varepsilon}{4}})$, there exists a constant $C_2>0$ such that $h^4\leq C_2\tau^{1+\varepsilon}$. Suppose
\begin{equation}\label{bb8}
	\|\xi_v^{*,n}\|^2 \leq Q_0(1+\tau)^{n}(\tau^2+h^4),\quad 1\leq n\leq n^*-1,
\end{equation}
for some $Q_0> e^{ \tilde QT}\tilde Q$,
and we intend to prove this relation for $n=n^*$ to complete the induction. As (\ref{bb8}) holds for $1\leq n\leq n^*-1$, we have $\|\widetilde{Pv(t_n)}-v_h^{*,n}\|^2\leq Q_0e^{T}(\tau^2+h^4)$, that is,
$$1-Q_0e^{T}(\tau^2+h^4) +\|v_h^{*,n}\|^2\leq  2(\widetilde{Pv(t_n)},v_h^{*,n}), $$
which implies $(\widetilde{Pv(t_n)},v_h^{*,n})\geq 0$ for $\tau$ and $h$ small enough. We divide this equation by $\|v_h^{*,n}\|$ to get $(\widetilde{Pv(t_n)},v_h^{n})\geq 0$, which implies $(\widetilde{Pv(t_n)},v_h^{n})+1\geq 1=\|\widetilde{Pv(t_n)}\|^2$, i.e. $\big(\widetilde{Pv(t_n)}, \widetilde{Pv(t_n)}-v_h^n\big)\leq 1$. Thus we could invoke Lemma \ref{li}, which implies $\|\xi_v^n\|^2\leq \|\xi_v^{*,n}\|^2+2\|\xi_v^{*,n}\|^4+\|\xi_v^{*,n}\|^6$ for $1\leq n\leq n^*-1$, in (\ref{zjy4}) to get
\begin{equation*}
	\begin{aligned}
\|\xi_v^{*,m}\|^2
&\leq \sum_{n=1}^{m}(2\|\xi_v^{*,n-1}\|^4+\|\xi_v^{*,n-1}\|^6)\\
&+\tilde Q\tau\sum_{n=1}^{m}(\|\xi_v^{*,n-1}\|^2+2\|\xi_v^{*,n-1}\|^4+\|\xi_v^{*,n-1}\|^6)+\tilde Q(\tau^2+h^4)\\
&\leq \tilde Q\tau\sum_{n=1}^{m}\|\xi_v^{*,n-1}\|^2+3\sum_{n=1}^{m}(\|\xi_v^{*,n-1}\|^4+\|\xi_v^{*,n-1}\|^6)+\tilde Q(\tau^2+h^4),
\end{aligned}
\end{equation*}
for $1\leq m\leq n^*$, where we used $\tilde Q\tau\leq 1/2$ for $\tau$ small enough.
An application of the Gronwall inequality yields
\begin{align}\label{zjy6}
\|\xi_v^{*,m}\|^2
&\leq 3e^{ \tilde QT}\sum_{n=1}^{m}(\|\xi_v^{*,n-1}\|^4+\|\xi_v^{*,n-1}\|^6)+e^{ \tilde QT}\tilde Q(\tau^2+h^4),~~1\leq m\leq n^*.
\end{align}
We invoke the induction hypothesis and  $h^4\leq C_2\tau^{1+\varepsilon}$ to get
\begin{align*}
	& \sum_{n=1}^{n^*} \|\xi_v^{*,n-1}\|^4\leq  Q_0^2(\tau^2+h^4)^2\sum_{n=1}^{n^*}(1+\tau)^{2(n-1)}=Q_0^2(\tau^2+h^4)^2\frac{(1+\tau)^{2n^*}-1}{(1+\tau)^2-1}\\
	&\quad\leq Q_0^2(\tau^2+h^4)^2\frac{(1+\tau)^{2n^*}}{2\tau}\leq Q_0^2(\tau^2+h^4)\frac{(\tau+C_2\tau^{\varepsilon})}{2}(1+\tau)^{n^*}e^{T};\\
	& \sum_{n=1}^{n^*} \|\xi_v^{*,n-1}\|^6\leq Q_0^3(\tau^2+h^4)^2\frac{(\tau+C_2\tau^{\varepsilon})}{3}(1+\tau)^{n^*}e^{2T}.
\end{align*}
We incorporate the above estimates in (\ref{zjy6}) with $m=n^*$  to get $\|\xi_v^{*,n^*}\|^2 \leq r(\tau) (1+\tau)^{n^*}(\tau^2+h^4)$ where
\begin{align*}
	r(\tau):=3e^{ \tilde QT}(\tau+C_2\tau^{\varepsilon})\Big(Q_0^2\frac{1}{2}e^{T}+Q_0^3(\tau^2+h^4)\frac{1}{3}e^{2T}\Big)+e^{ \tilde QT}\tilde Q .
\end{align*}
As $r(\tau)$ is a continuous function of $\tau$ with $r(0)=e^{ \tilde QT}\tilde Q<Q_0$, then for $\tau$ small enough we have $r(\tau)\leq Q_0$, which, together with the above estimate, leads to (\ref{bb8}) with $n=n^*$. Thus (\ref{bb8}) holds for $1\leq n\leq N$ by induction, leading to $\|\xi_v^{*,n}\|^2 \leq Q(\tau^2+h^4)$ for $1\leq n\leq N$. We invoke this in Lemma \ref{li} to get
$\|\xi_v^{n}\|^2 \leq Q(\tau^2+h^4)$, which, together with the estimate of $\eta_v^n$ in (\ref{bb9}), gives $\|e_v^{n}\|^2 \leq Q(\tau^2+h^4)$. We invoke this back to (\ref{bb5}) to get $\|\xi_u^{n}\| \leq Q(\tau+h^2)$, which, together with $\|\eta_u^n\|\leq Qh^2$, leads to $\|e_u^{n}\| \leq Q(\tau+h^2)$. The proof is thus completed.
\end{proof}

\section{Numerical experiments} We present numerical examples to show the performance of the scheme (\ref{app}) and substantiate the theoretical results. The uniform partitions are used for both time period $[0,T]$  with $T=5$ and space domain $(0,\pi)^d$ with time step size $\tau=T/N$ and spatial mesh size $h=\pi/M$ for some integers $N$, $M>0$. The errors are measured by $\text{Err}(u):=\max\limits_{1\leq n\leq N}\|u(t_n)-u_h^n\|$ and $\text{Err}(v):=\max\limits_{1\leq n\leq N}\|v(t_n)-v_h^n\|$, and we define the vector $F(u_h(T))\in\mathbb R^{(M-1)^d}$ as  $F(u_h(T))_i:=(\Delta_h u_h(T))(x_i)+f(u_h(x_i,T))$ for $1\leq i\leq (M-1)^d$. Here $x_i$ denotes the $i$th internal node and $\Delta_h$ denotes the discrete Laplacian determined from the finite element scheme $(\Delta_h u_h(T),\chi)=-(\nabla u_h(T),\nabla\chi)$ for any $\chi\in S_h$ \cite[Equation 1.33]{Tho}. Thus $F(u_h(T))$ is used to measure the degree of $u_h(T)$ in satisfying the equation (and thus in approximating the saddle point). We always set $\beta =\gamma =1$.

 \textbf{Example 1: One-dimensional case.} We first adopt the example in \cite{xie05}, which corresponds to model (\ref{elliptic}) with $d=1$ and $f(u)=u^4-10u^2$.
We test the performance of the scheme (\ref{app}) under $\tau=10^{-3}$, $h=\pi/5000$, $v_0=\sin x$ and different $u_0$ in Fig. \ref{fig1} (a)--(b), which shows that the method successfully converges to two different saddle points from two different initial guesses. The $\|F(u_h(T))\|_{l^\infty}$ for cases (a)--(b) are $[3.50, 7.00]\times 10^{-4}$, respectively,  and the first two smallest eigenvalues of the Hessian at $u_h(T)$ are $[-111.27,0.02]$ for case (a) and $[-0.27, 7.64]$ for case (b), indicating that the scheme (\ref{app}) preserves the indexes of saddle points. 

We also pick $f(u)=4.1^2 \sin u$ based on the suggestions at the bottom of Remark \ref{rem33}, under which  the hypothesis for the index-preservation issue in Corollary \ref{cor1} is valid. We select the same parameters as above, and different initial values $u_0$  are used in Fig. \ref{fig1} (c)--(d), which demonstrates that the method successfully converges to two different index-1 saddle points from two distinct initial guesses. The $\|F(u_h(T))\|_{l^\infty}$ for cases (c)--(d) are $[1.10, 0.30]\times 10^{-3}$, respectively,  and the first two smallest eigenvalues of the Hessian at $u_h(T)$ are $[-0.23, 16.37]$ for case (c) and $[-1.49, 3.52]$ for case (d), again indicating that the scheme (\ref{app}) preserves the indexes of saddle points.

\vspace{-0.2in}

\begin{figure}[H]
\centering
\subfloat[$u_0=\sin x$]{\label{figa1} 
    \includegraphics[width=0.4\linewidth]{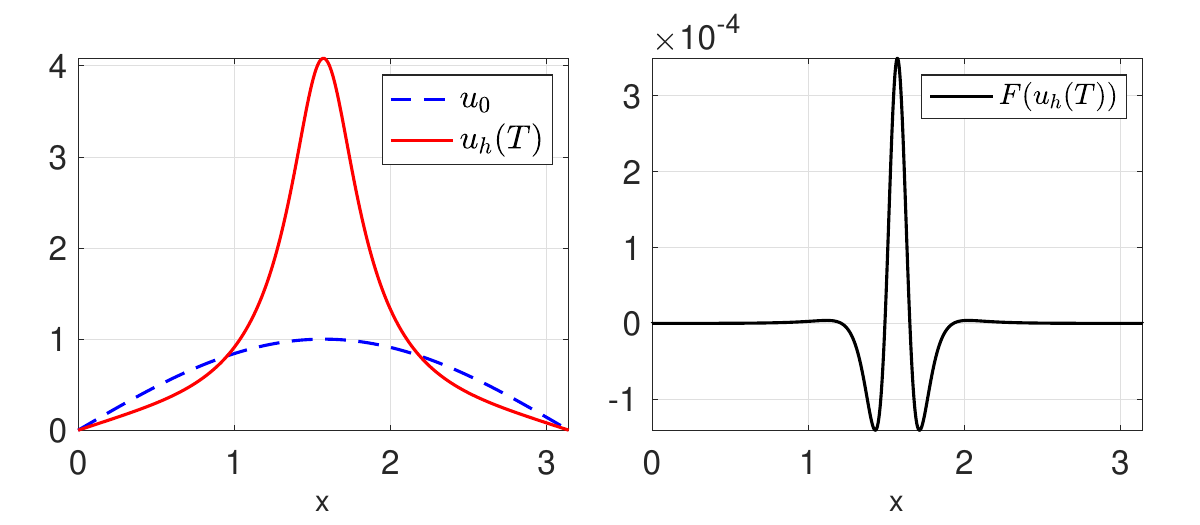}} 
    %\hfill 
    \hspace{0.2in}
\subfloat[$u_0=-3\sin 3x$]{\label{figa2}
    \includegraphics[width=0.4\linewidth]{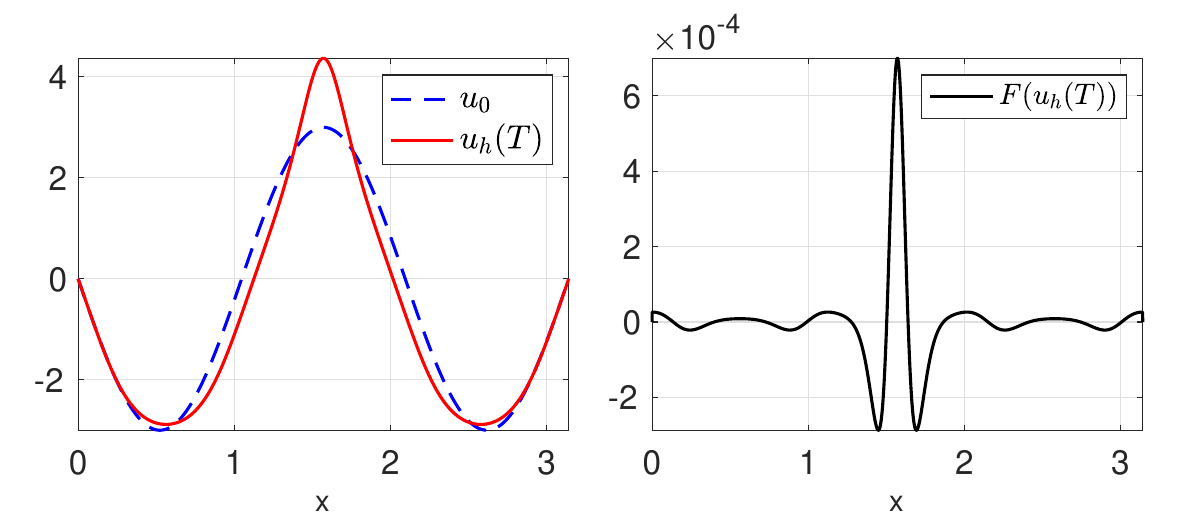}}
\vspace{-0.1in} 
\subfloat[$u_0=0.1x(\pi-x)^2$]{\label{figa3}
    \includegraphics[width=0.4\linewidth]{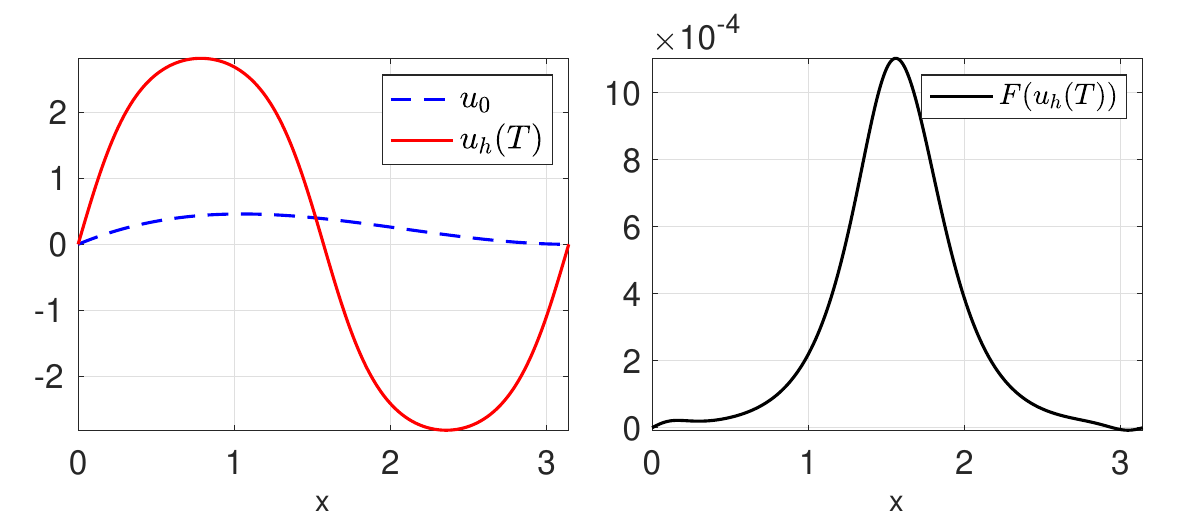}}
%\hfill
 \hspace{0.2in}
\subfloat[$u_0=3\sin x + \sin 3x$]{\label{figa4}
    \includegraphics[width=0.4\linewidth]{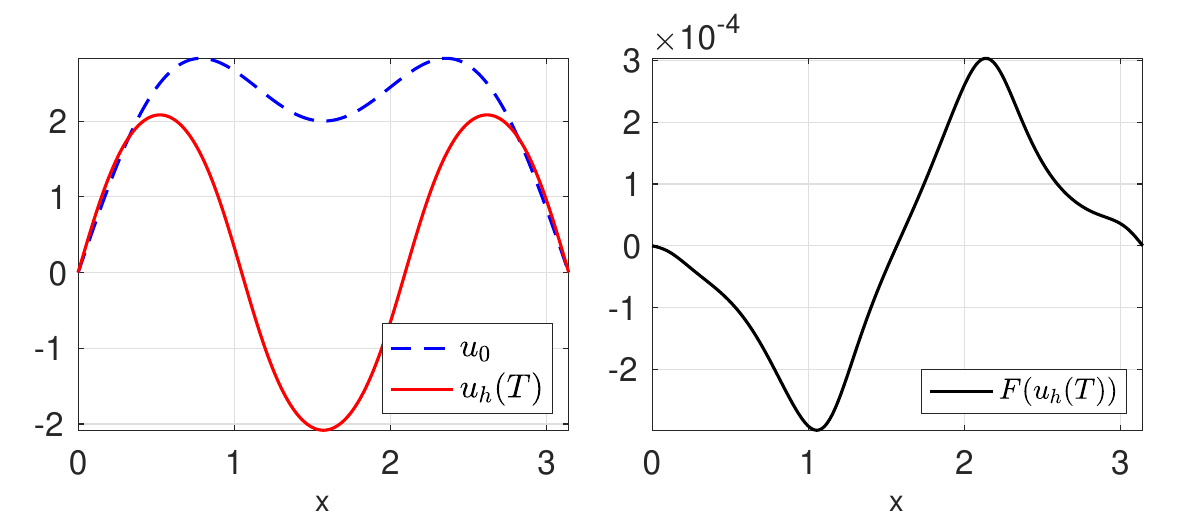}}
\caption{Plots of $u_h(T)$ and $F(u_h(T))$ for Example 1.}
\label{fig1}
\end{figure}
\vspace{-0.2in} 
%Then we modify $u_0$ in Fig. \ref{fig1}, leading to a different index-1 saddle point in Fig. \ref{fig2}, which has also been observed in \cite[Chapter 8]{xie05}. The values of $\|F(u_h(T))\|_{l^\infty}$ for cases (a)--(d) become $[1.40,1.51,2.02,6.88]\times 10^{-3}$, respectively,  and the first two smallest eigenvalues of the Hessian at $u_h(T)$ become $[-111.27,0.02]$ for all cases, leading to the same conclusions as above. 
%
%\vspace{-0.2in}
%
%\begin{figure}[H]
%\centering
%\subfloat[$u_0=\sin x$]{\label{b1}
%    \includegraphics[width=0.4\linewidth]{e11.pdf}} 
%   %\hfill 
%    \hspace{0.2in}
%\subfloat[$u_0=-0.1x\sin 2x$]{\label{b2}
%    \includegraphics[width=0.4\linewidth]{e12.pdf}}
%\vspace{-0.1in} 
%\subfloat[$u_0=0.5\sqrt{x}(\pi-x)$]{\label{b3}
%    \includegraphics[width=0.4\linewidth]{b3.pdf}}
%%\hfill 
%    \hspace{0.2in}
%\subfloat[$u_0=(2-x)\sin x \ln(1+\cos x)$]{\label{b4}\includegraphics[width=0.4\linewidth]{e11.pdf}}
%\caption{Plots of $u_h(T)$ and $F(u_h(T))$ for Example 1.}
%\label{fig2}
%\end{figure}

To test the convergence rates, the reference solution is computed with $\tau=10^{-4}$ and $h=\pi/10000$. We fix $\tau =10^{-4}$ and select different $h$ to test the spatial convergence, and we fix $h=\pi/10000$ and select different $\tau$ when testing the temporal convergence. Tables \ref{conv1}--\ref{conv2} demonstrate the first and second order accuracy in time and space, respectively, confirming theoretical predictions. 

\vspace{-0.2in}

\begin{table}[H]
\caption{Accuracy tests for Example 1 case (a).}
\centering
\footnotesize{\begin{tabular}{ccccc|ccccc}
\hline
$\tau$ & $\text{Err}(u)$ & rate   & $\text{Err}(v)$ & rate & $h$ & $\text{Err}(u)$ & rate   & $\text{Err}(v)$ & rate\\\hline
$1.6\times 10^{-2}$  & 7.55e-03  &      &   2.42e-03 & & $\pi/200$  & 3.23e-05  &      &   2.06e-05 & \\
$8\times 10^{-3}$    & 3.95e-03	 & 0.93 &   1.27e-03 & 0.93 & $\pi/400$  & 8.06e-06  & 2.00 &   5.13e-06 & 2.00\\
$4\times 10^{-3}$    & 1.96e-03  & 1.01 &   6.29e-04 & 1.01 & $\pi/800$  & 2.00e-06  & 2.01 &   1.28e-06 & 2.01\\ 
$2\times 10^{-3}$    & 9.55e-04  & 1.04 &   3.06e-04 & 1.04 & $\pi/1600$ & 4.91e-07  & 2.03 &   3.13e-07 & 2.03\\ \hline
\end{tabular}}
\label{conv1}
\end{table}

\vspace{-0.3in}
\begin{table}[H]
 \caption{Accuracy tests for Example 1 case (c).}
 \centering
\footnotesize{\begin{tabular}{ccccc|ccccc}
\hline
$\tau$ & $\text{Err}(u)$ & rate   & $\text{Err}(v)$ & rate & $h$ & $\text{Err}(u)$ & rate   & $\text{Err}(v)$ & rate\\\hline
$1.6\times 10^{-2}$  & 8.76e-04  &      &   1.92e-03 & & $\pi/200$  & 9.45e-06  &      &    1.88e-05 & \\
$8\times 10^{-3}$    & 4.31e-04	 & 1.02 &   9.64e-04 & 0.99 & $\pi/400$  & 2.38e-06  & 1.99 &   4.72e-06 & 1.99\\
$4\times 10^{-3}$    & 2.14e-04  & 1.01 &   4.91e-04 & 0.97 & $\pi/800$  & 6.02e-07  & 1.98 &   1.18e-06 & 2.00\\ 
$2\times 10^{-3}$    & 1.06e-04  & 1.01 &   2.45e-04 & 1.00 & $\pi/1600$ & 1.54e-07  & 1.97 &   2.96e-07 & 2.00\\ \hline
\end{tabular} } 
\label{conv2}
\end{table}

\textbf{Example 2: Two-dimensional case.} We consider model (\ref{elliptic}) with $d=2$ and $f(u)=u^3$ or $u^5$, which corresponds to the Lane-Emden equation. We test the performance of the scheme (\ref{app}) under $\tau=10^{-3}$, $h=\pi/200$, $v_0=\sin x_1 \sin x_2$ and different $f(u)$ in Fig. \ref{fig2}, which shows that the method successfully converges to the target saddle point. The $\|F(u_h(T))\|_{l^\infty}$ for cases (a)--(b) are $[5.13, 9.67]\times 10^{-3}$, respectively,  and the first two smallest eigenvalues of the Hessian at $u_h(T)$ are $[-2.09,2.09]$ for case (a) and $[-4.66, 1.73]$ for case (b), indicating that the scheme (\ref{app})  preserves the indexes of saddle points. 

\vspace{-0.1in}
\begin{figure}[H]
\centering
\subfloat[$f(u)=u^3$ and $u_0=\sin x_1 \sin x_2$]{
		\includegraphics[width=0.8\linewidth]{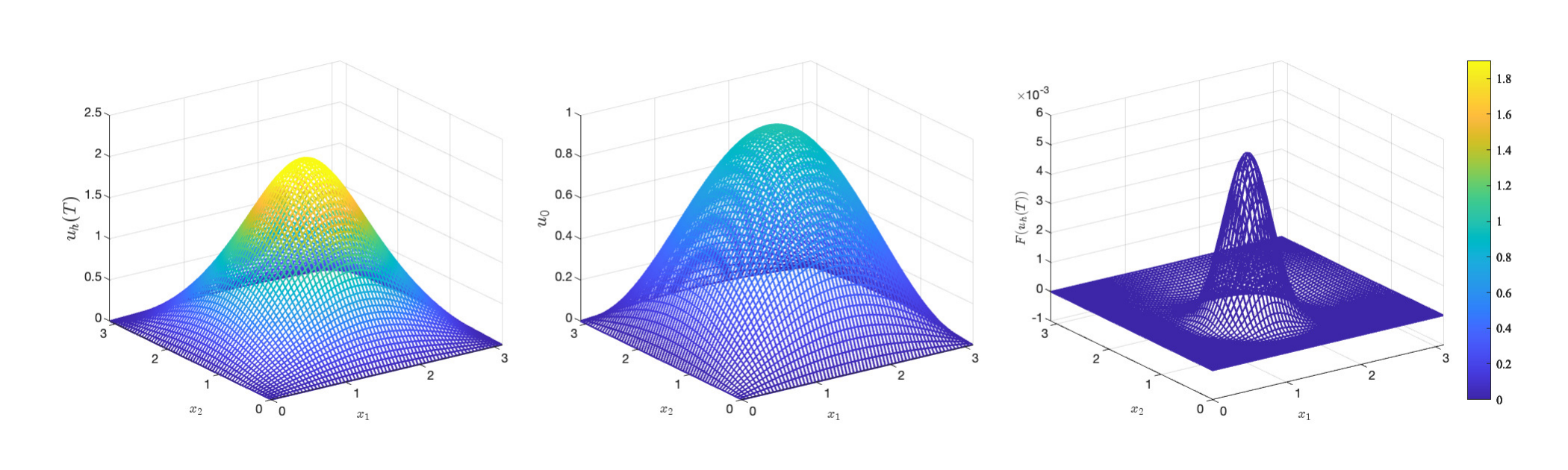}}\\
\vspace{-0.15in}
\subfloat[$f(u)=u^5$ and $u_0=x_1(\pi-x_2)\tan \frac{(x_1-\pi)}{4}\sin 4x_2$]{
		\includegraphics[width=0.8\linewidth]{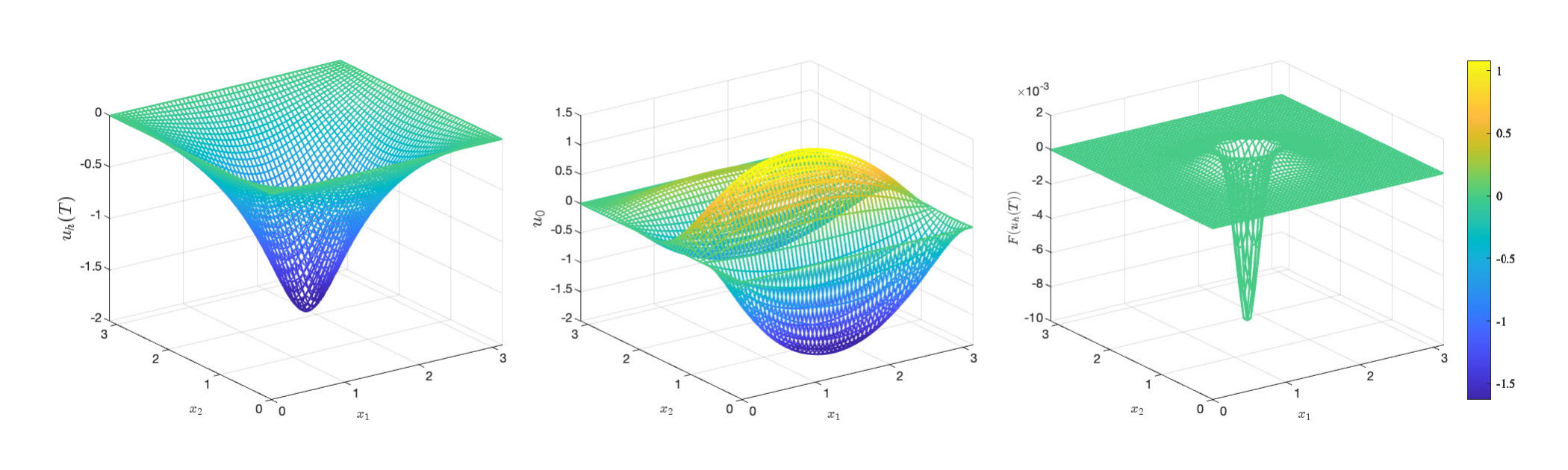}}
\caption{Plots of $u_h(T)$ and $F(u_h(T))$ for Example 2.}
\label{fig2}
\end{figure}

\vspace{-0.4in}

\section{Concluding remarks}
In this paper we analyze the I-1 SD for locating transition states of the semilinear elliptic problem, which takes the form of a coupled parabolic system. In particular, we prove the well-posedness, $H^1$ stability, and error estimates of semi- and fully-discrete finite element schemes. Naturally, we may consider extending the current work to the index-$k$ case, which turns to study a coupled system with the state variable $u$ and auxiliary variables $\{v_1,\cdots,v_k\}$ that satisfy orthonormal conditions. Due to the more complicated coupling and constraints, a substantial modification is required with enhanced ideas and techniques. We will address this interesting topic in the subsequent work.

%\bibliographystyle{siamplain}
%\bibliography{references}

\vspace{-0.1in}
\begin{thebibliography}{10}
\vspace{-0.1in}
\bibitem{Ada} R. Adams and J. Fournier, {\it Sobolev spaces}, Elsevier, San Diego, 2003.

\bibitem{Bao} W. Bao and Q. Du, Computing the ground state solution of Bose-Einstein condensates by a normalized gradient flow.
{\it SIAM J. Sci. Comput.} 25 (2004), 1674--1697.

\bibitem{Bre} S. Brenner and L. Scott, {\it The mathematical theory of finite element methods}, Third Ed., Springer, New York, 2007.

%\bibitem{BaoCai} W. Bao and Y. Cai, Mathematical theory and numerical methods for Bose-Einstein condensation.{\it Kinet. Relat. Models} 6 (2013), 1--135.

%\bibitem{xie04} C. Chen and Z. Xie, Search extension method for multiple solutions of a nonlinear problem, Comput. Math. Appl. 47 (2004), 327--343.

\bibitem{xie05} C. Chen and Z. Xie, Structure of multiple solutions for nonlinear differential equations. {\it Sci. China Ser. A-Math.} 47 (2004), 172--180.

%\bibitem{xie06} C. Chen and Z. Xie, {\it Search-extension method for multiple solution computation of nonlinear differential equations}, Science Press, Beijing, 2005.

%\bibitem{BaoCao} F. Bao, Y. Cao, A. Meir, W. Zhao, A first order scheme for backward doubly stochastic differential equations. {\it SIAM/ASA J. Uncertain. Quantif.} 4 (2016), 413--445.

%\bibitem{Bao} W. Bao, Y. Cai, Mathematical theory and numerical methods for Bose-Einstein condensation. {\it Kinetic and Related Models} 6 (2013), 1--135.

%\bibitem{BaoDu} W. Bao, Q. Du, Y. Zhang, Dynamics of rotating Bose-Einstein condensates and its efficient and accurate numerical computation. {\it SIAM J. Appl. Math.} 66 (2006), 758--786.

\bibitem{Cho} Y. Choi and P. McKenna, A mountain pass method for the numerical solution of semilinear elliptic problems. {\it Nonlinear Anal.} 20 (1993), 417--437.

\bibitem{Din} Z. Ding, D. Costa and G. Chen, A high-linking algorithm for sign-changing solutions of semilinear elliptic equations. {\it  Nonlinear Anal.} 38 (1999), 151--172.
%\bibitem{Doye} J. Doye and D. Wales, Saddle points and dynamics of Lennard-Jones clusters, solids, and supercooled liquids. {\it J. Chem. Phys.} 116 (2002), 3777--3788.


%\bibitem{DuLi} Q. Du, R. Li, L. Zhang, Variational phase field formulations of polarization and phase transition in ferroelectric thin films. {\it SIAM J Appl Math} 80 (2020), 1590--1606.

%\bibitem{Eck} B. Eckhardt, Irregular scattering. {\it Phys. D} 33 (1988), 89--98.

%\bibitem{EV2010}{W. E, E. Vanden-Eijnden},  {Transition-path theory and path-finding algorithms for the study of rare events},  {\it Annu. Rev. Phys. Chem.}, 61 (2010), 391-420.
%{{blue}{\bibitem{Evans}L. Evans, Partial Differential Equations, American Mathematical Society, Berkeley, 2010.}}


\bibitem{EZho} W. E and X. Zhou, The gentlest ascent dynamics. {\it Nonlinearity} 24 (2011), 1831--1842.


\bibitem{Farr} P. Farrell, \'{A}. Birkisson and S. Funke, Deflation techniques for finding distinct solutions of nonlinear partial differential equations. {\it SIAM J. Sci. Comput.} 37 (2015), A2026--A2045.

\bibitem{Gao} W. Gao, J. Leng and X. Zhou, An iterative minimization formulation for saddle point search. {\it SIAM J. Numer. Anal.} 53 (2015), 1786--1805.

%\bibitem{goodfellow2016deep} I. Goodfellow, Y. Bengio, A. Courville, {\it Deep learning}. The MIT Press, Cambridge, MA, 2016.
  
%\bibitem{Gou} N. Gould, C. Ortner and D. Packwood, A dimer-type saddle search algorithm with preconditioning and linesearch. {\it Math. Comp.} 85 (2016), 2939--2966.

%\bibitem{Gra}  W. Grantham, Gradient transformation trajectory following algorithms for determining stationary min-max saddle points, in Advances in Dynamic Game Theory, Ann. Internat. Soc. Dynam. Games 9, Birkhauser Boston, Boston, MA, 2007, 639--657.

\bibitem{GuZho} S. Gu and X. Zhou, Simplified gentlest ascent dynamics for saddle points in non-gradient systems. {\it Chaos} 28 (2018), 123106.

%\bibitem{Haibook} E. Hairer, C. Lubich, G. Wanner, Geometric numerical integration: Structure-preserving algorithms for ordinary differential equations, 2nd edn., Springer, Berlin, 2006.

%\bibitem{Han2019transition} Y. Han, Y. Hu, P. Zhang, A. Majumdar, L. Zhang, {Transition pathways between defect patterns in confined nematic liquid crystals}. {\it J. Comput. Phys.} 396 (2019), 1--11.

%\bibitem{HanXu} Y. Han, Z. Xu, A. Shi, L. Zhang, Pathways connecting two opposed bilayers with a fusion pore: a molecularly-informed phase field approach. {\it Soft Matter} 16 (2020), 366--374.

%\bibitem{HanYin} Y. Han, J. Yin, P. Zhang, A. Majumdar, L. Zhang, Solution landscape of a reduced Landau--de Gennes model on a hexagon. {\it Nonlinearity} 34 (2021), 2048--2069.

%\bibitem{Han2021} Y. Han, J. Yin, Y. Hu, A. Majumdar, L. Zhang, Solution landscapes of the simplified Ericksen-Leslie model and its comparison with the reduced Landau-de Gennes model, {\it Proceedings of the Royal Society A}, 477 (2021), 20210458.

\bibitem{Hao} W. Hao, J. Hesthaven, G. Lin, and B. Zheng, A homotopy method with adaptive basis selection for computing multiple solutions of differential equations.
{\it J. Sci. Comput.} 82 (2020), 19.

\bibitem{Dimer} G. Henkelman and H. J{\'o}nsson, A dimer method for finding saddle points on high dimensional potential surfaces using only first derivatives. {\it J. Chem. Phys.} 111 (1999), 7010--7022.
  
 \bibitem{Hyt} T. Hyt\"onen, J. van Neerven, M. Veraar, and L. Weis, Bochner spaces. In: Analysis in Banach Spaces . Ergebnisse der Mathematik und ihrer Grenzgebiete. 3. Folge / A Series of Modern Surveys in Mathematics, vol 63. Springer, Cham, 2016.
  
%\bibitem{Hen} G. Henkelman and H. Jonsson, Improved tangent estimate in the nudged elastic band method for finding minimum energy paths and saddle points. {\it J. Chem. Phys.} 113 (2000), 9978--9985.

%\bibitem{Hen2} G. Henkelman, B. Uberuaga, and H. J\'onsson, A climbing image nudged elastic band method for finding saddle points and minimum energy paths. {\it J. Chem. Phys.} 113 (2000), 9901--9904.

\bibitem{Lev} A. Levitt and C. Ortner, Convergence and cycling in walker-type saddle search algorithms. {\it SIAM J. Numer. Anal.} 55 (2017), 2204--2227.

\bibitem{Lar} S. Larsson,
The long-time behavior of finite-element approximations of solutions to semilinear parabolic problems. {\it SIAM J. Numer. Anal.} 26 (1989), 348--365.

\bibitem{Li} B. Li, A bounded numerical solution with a small mesh size implies existence of a smooth solution to the Navier-Stokes equations.
{\it Numer. Math.} 147 (2021), 283--304.

\bibitem{LiQi} J. Li and J. Qi,  Continuous dependence of eigenvalues on potential functions for nonlocal Sturm--Liouville equations. {\it Math. Meth. Appl. Sci.} 46 (2023), 10617--10623.

\bibitem{Li2001} Y. Li and J. Zhou, A minimax method for finding multiple critical points and its applications to semilinear PDEs. {\it SIAM J. Sci. Comput.} 23 (2001), 840--865.

\bibitem{Li2002}  Y. Li and J. Zhou, Convergence results of a local minimax method for finding multiple critical points. {\it SIAM J. Sci. Comput.} 24 (2002), 865--885.

\bibitem{LiJi}  Z. Li, B. Ji and J. Zhou, A local minimax method using virtual geometric objects: Part I---For finding saddles.
{\it J. Sci. Comput.} 78 (2019), 202--225.

%\bibitem{LiZho} Z. Li and J. Zhou, A local minimax method using virtual geometric objects: Part II---For finding equality constrained saddles. {\it J. Sci. Comput.} 78 (2019), 226--245.

\bibitem{LiuXie} W. Liu, Z. Xie and W. Yi, Normalized Wolfe-Powell-type local minimax method for finding multiple unstable solutions of nonlinear elliptic PDEs. {\it  Sci. China Math.} 66 (2023), 2361--2384.

\bibitem{Liu2} W. Liu, Z. Xie and W. Yi, Normalized Goldstein-type local minimax method for finding multiple unstable solutions of semilinear elliptic PDEs. {\it Commun. Math. Sci.} 19 (2021), 147--174.

%\bibitem{Luo} Y. Luo, X. Zheng, X. Cheng, L. Zhang, Convergence analysis for discrete high-index saddle dynamics. {\it SIAM J. Numer. Anal.}, 60 (2022), 2731--2750.

%\bibitem{Mehta} D. Mehta, Finding all the stationary points of a potential-energy landscape via numerical polynomial-homotopy-continuation method, {\em Phys. Rev. E} 84 (2011), 025702. 

\bibitem{MiaCSIAM} S. Miao, L. Zhang, P. Zhang and X. Zheng, Construction and analysis for orthonormalized Runge--Kutta schemes of high-index saddle dynamics. {\it Commun. Nonlinear Sci. Numer. Simul.} 145 (2025), 108731.

%{{blue}\bibitem{Mil} G. Milstein and M. Tretyakov, Numerical algorithms for semilinear parabolic equations with small parameter based on approximation of stochastic equations. {\it Math. Comp.} 69 (1999), 237--267.}


%\bibitem{Milnor} J. W. Milnor, {\em Morse Theory}, Princeton University Press, 1963.

%\bibitem{Nes} Y. Nesterov, {\em Introductory lectures on convex optimization: a basic course},  Springer Science \& Business Media vol. 87, 2003.

%\bibitem{Tho} J. Thomson, On the structure of the atom: An investigation of the stability and periods of oscillation of a number of corpuscles arranged at equal intervals around the circumference of a circle; with application of the results to the theory of atomic structure. {\it London, Edinburgh, Dublin Phil. Mag. J. Sci.}, 7 (1904), 237--265.

%\bibitem{Min} R. M. Minyaev, W. Quapp, G. Subramanian, P. R. Schleyer, Y. Ho, Internal conrotation and disrotation in H2BCH2BH2 and diborylmethane 1, 3 H exchange. {\it J. Comput. Chem.} 18 (1997), 1792--1803.

\bibitem{Qua} W. Quapp and J. Bofill, Locating saddle points of any index on potential energy surfaces by the generalized gentlest ascent dynamics. {\it Theor. Chem. Acc.} 133 (2014), 1510.

\bibitem{Shi} B. Shi, Y. Han, C. Ma, A. Majumdar and L. Zhang, A modified Landau-de Gennes theory for smectic liquid crystals: Phase transitions and structural transitions. {\it SIAM J. Appl. Math.} 85 (2025), 821--847.

\bibitem{Shi2} B. Shi, Y. Han, A. Majumdar and L. Zhang, Multistability for nematic liquid crystals in cuboids with degenerate planar boundary conditions. {\it SIAM J. Appl. Math.} 84 (2024), 756--781.

\bibitem{Tho} V. Thom\' ee, {\it Galerkin finite element methods for parabolic problems}. Springer, Berlin Heidelberg, 2006.

\bibitem{wang2021modeling} W. Wang, L. Zhang and P. Zhang, Modelling and computation of liquid crystals. {\it Acta Numerica} 30 (2021), 765--851.
  
\bibitem{Wan} Z. Wang and J. Zhou, A local minimax-Newton method for finding multiple saddle points with symmetries.
{\it SIAM J. Numer. Anal.} 42 (2004), 1745--1759.

\bibitem{Wan2} Z. Wang and J. Zhou, An efficient and stable method for computing multiple saddle points with symmetries. {\it
SIAM J. Numer. Anal.} 43 (2005), 891--907.

\bibitem{Xie0}  Z. Xie, Y. Yuan and J. Zhou, On finding multiple solutions to a singularly perturbed Neumann problem. {\it SIAM J. Sci. Comput.} 34 (2012), A395--A420.

\bibitem{Xie} Z. Xie, Y. Yuan and J. Zhou, On solving semilinear singularly perturbed Neumann problems for multiple solutions. {\it SIAM J. Sci. Comput.} 44 (2022),
A501--A523.

\bibitem{XieIMA} Z. Xie and C. Chen, An improved search-extension method for solving multiple solutions of semilinear PDEs. {\it IMA J. Numer. Anal.} 25 (2005), 549--576.

%\bibitem{XTKW2014} {X. Xu, C. L. Ting, I. Kusaka, Z. Wang},  {Nucleation in polymers and soft matter}. {\it Annu. Rev. Phys. Chem.}, 65 (2014), 449-475.

%\bibitem{Xu_PRE} Z. Xu, Y. Han, J. Yin, B. Yu, Y. Nishiura, L. Zhang, Solution landscapes of the diblock copolymer-homopolymer model under two-dimensional confinement. {\it Phys. Rev. E} 104 (2021), 014505.

%\bibitem{CHiSD2021} J. Yin, Z. Huang, L. Zhang, Constrained high-index saddle dynamics for the solution landscape with equality constraints, {\it J. Sci. Comput.} 91 (2022), 62.

\bibitem{Yin2020nucleation} J. Yin, K. Jiang, A.-C. Shi, P. Zhang and L. Zhang, {Transition pathways connecting crystals and quasicrystals}, {\it Proc. Natl. Acad. Sci. U.S.A.}, 118 (2021), e2106230118.

\bibitem{YinPRL} J. Yin, Y. Wang, J. Chen, P. Zhang and L. Zhang, Construction of a pathway map on a complicated energy landscape. {\it Phys. Rev. Lett.} 124 (2020), 090601.

%\bibitem{YinSCM} J. Yin, B. Yu, L. Zhang, Searching the solution landscape by generalized high-index saddle dynamics. {\it Sci. China Math.} 64 (2021), 1801.

\bibitem{Yao} X. Yao and J. Zhou, Unified convergence results on a minimax algorithm for finding multiple critical points in Banach spaces.
{\it SIAM J. Numer. Anal.} 45 (2007), 1330--1347.

\bibitem{YinSISC} J. Yin, L. Zhang and P. Zhang, High-index optimization-based shrinking dimer method for finding high-index saddle points. {\it SIAM J. Sci. Comput.} 41 (2019), A3576--A3595.


%\bibitem{Yin2021} J. Yin, L. Zhang, P. Zhang, {Solution landscape of Onsager functional identifies non-axisymmetric critical points}, {\it Physica D: Nonlinear Phenomena} 430 (2022), 133081.

%\bibitem{YuZhaZha} B. Yu, X. Zheng, P. Zhang, L. Zhang, Computing solution landscape of nonlinear space-fractional problems via fast approximation algorithm. {\it J. Comput. Phys.} 468 (2022), 111513.
      
%\bibitem{ZhaDu} J. Zhang and Q. Du, Shrinking dimer dynamics and its applications to saddle point search. {\it SIAM J. Numer. Anal.} 50 (2012), 1899--1921.

%\bibitem{ZhaDuJCP}  J. Zhang, Q. Du, Constrained shrinking dimer dynamics for saddle point search with constraints. {\it J. Comput. Phys.} 231 (2012), 4745--4758.

%\bibitem{npj2016} L. Zhang, W. Ren, A. Samanta, Q. Du, Recent developments in computational modelling of nucleation in phase transformations. {\it npj Comput. Mater.} 2 (2016), 16003.
  
%\bibitem{ZhangChe} L. Zhang, L. Chen, Q. Du, Morphology of critical nuclei in solid-state phase transformations. {\it Phys. Rev. Lett.} 98 (2007), 265703.

%\bibitem{ZhangCheDu} L. Zhang, L. Chen, Q. Du, Simultaneous prediction of morphologies of a critical nucleus and an equilibrium precipitate in solids. {\it Commun. Comput. Phys.} 7 (2010), 674--682.

\bibitem{ZhaSISC} L. Zhang, Q. Du and Z. Zheng, Optimization-based shrinking dimer method for finding transition states. {\it SIAM J. Sci. Comput.} 38 (2016), A528--A544.

\bibitem{Z3} L. Zhang, P. Zhang and X. Zheng, Error estimates of Euler discretization to high-index saddle dynamics. {\it SIAM J. Numer. Anal.} 60 (2022), 2925--2944.

 \bibitem{Z3CSIAM} L. Zhang, P. Zhang and X. Zheng, Mathematical and numerical analysis to shrinking-dimer saddle dynamics with local Lipschitz conditions. {\it CSIAM Trans. Appl. Math.} 4 (2023), 157--176. 
 
%\bibitem{Z3c} L. Zhang, P. Zhang, X. Zheng, Discretization and index-robust error analysis for constrained high-index saddle dynamics on the high-dimensional sphere. {\it Sci. China Math.}  66 (2023), 2347--2360.

\bibitem{ZheSICON} X. Zheng and H. Wang, A hidden-memory variable-order fractional optimal control model: analysis and approximation. {\it SIAM J. Control Optim.} 59 (2021), 1851--1880.

\bibitem{Zho} J. Zhou, Solving multiple solution problems: Computational methods and theory revisited. {\it Commun. Appl. Math. Comput.} 31 (2017), 1--31.


\end{thebibliography}
\end{document}